\documentclass[10pt,a4paper]{article}
\usepackage[latin1]{inputenc}
\usepackage{amsmath}
\usepackage{amsfonts}
\usepackage{amssymb}
\usepackage{amsthm}
\usepackage{graphicx}

\theoremstyle{plain}
\newtheorem{theorem}{Theorem}
\newtheorem{lemma}{Lemma}

\theoremstyle{definition}

\theoremstyle{remark}

\def\A{\mathcal{A}}

\def\nbhd{neighborhood}

\def\rnd#1{} 
\def\R{{\mathbb R}}

\def\A{{\mathcal{A}}}

\def\Cl{\mathrm{Cl}}

\def\Int{\mathrm{Int}}

\def\MB{distancelike}

\author{Henry C. King}
\title{Tico Spines}

\begin{document}

\maketitle

\section{Introduction}

If $X$ is a compact manifold with boundary, a \emph{spine} of $X$ is a subset $Y$ of the interior of $X$
so that $X-Y$ is diffeomorphic to $\partial X\times [0,\infty)$\footnote{In other contexts one might ask for 
	the stronger condition that $X$ be a regular \nbhd\ of its spine, but we use this weaker notion in this paper.}.
A \emph{tico} $\A$ in $X$ is a finite collection $\A=\{A_1,\ldots,A_k\}$ of smooth codimension one closed
submanifolds\footnote{Outside the context of this paper, we would want to allow these submanifolds to have boundary and to intersect $\partial X$.  But here we will not allow such behavior.} of the interior of $X$ in general position.
The \emph{realization} $|\A|$ of the tico is the subset $\bigcup _{i=1}^k A_i$.
In various published works, Akbulut and I stated the following theorem.

\begin{theorem}\label{main_thm}
	Suppose $M$ is a smooth compact manifold which bounds a smooth compact manifold $X$.
	Then there is a smooth compact manifold $Z$ and a tico $\A$ in $Z$  so that $|A|$ is a spine of $Z$.
\end{theorem}

For its proof we relied on a picture like Figure 1
at a crucial point and gave no details.
This paper is meant to provide the details.
\begin{figure}
\centering
\includegraphics[scale=0.4]{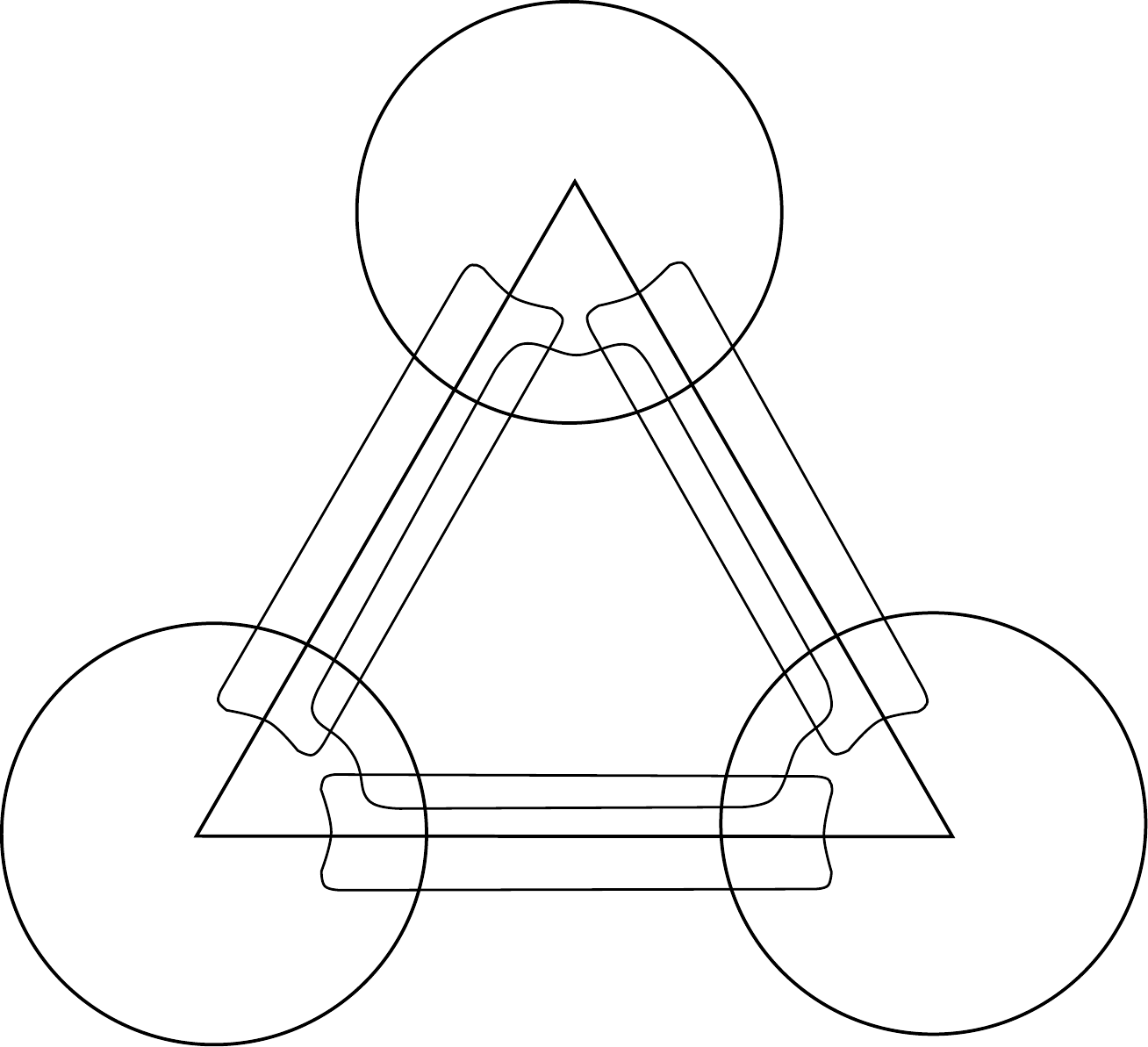}
\caption{This paper in a Nutshell}
\end{figure}

\section{Dimension 2}

In an earlier version of this paper, the general proof omitted the case $\dim M=2$
so an ad hoc proof for $\dim M=2$ was needed.
This is no longer the case but I leave in this section anyway, if only to provide concrete examples.

By classification of surfaces, we know that if the surface $M$ bounds, it can be obtained by attaching
some one handles to a two sphere.
Take the standard unit sphere $S^2$ in $\R^3$, bounding the unit ball $B^3$.
Attach one handles to $B^3$ to obtain a 3 manifold $Y$ whose boundary is $M$.
When we attach a one handle, we however are careful to attach it to two antipodal points on the sphere.
For each one handle $[0,1]\times D^2$ attached we consider the following curve $C_i$ in $Y$ consisting of the core 
$[0,1]\times 0$ of the one handle, union the line segment between the two attaching points.
Since the attaching points are antipodal, these curves all go through the origin.
Note that the union of the $C_i$ is a spine of $Y$.
Let $\pi\colon Y'\to Y$ be the blow up of the origin.
Then $Q = \pi^{-1}(\bigcup C_i)$ is a spine of $Y'$.
Note $Q$ consists of the projective plane $\pi^{-1}(0)$ union a bunch of disjoint curves $C'_i$, the strict transforms of the $C_i$.
These curves are transverse to the exceptional divisor $\pi^{-1}(0)$.
Now let $\rho\colon X\to Y'$ be the blow up of the union of the $C'_i$.
We have a tico $\A$ in $X$ consisting of the strict transform of $\pi^{-1}(0)$ and the $\rho^{-1}(C_i')$.
Note $|\A|$ is a spine of $X$, since the blowup maps restrict to a diffeomorphism of $X-|\A|$ to $Y-\bigcup C_i$.

\section{Terminology}

Some terminology I introduce in this paper is meant to be
completely local to this paper, as it is introduced
solely to work with the particular proof in this paper.
Easy generalizations may be suppressed, no attempt
has been made to reconcile with the literature, etc.
It is meant solely to get the job done with reasonable
efficiency, with occasional excursions to side issues
I find amusing.
But the following should be reasonably standard.

Smooth means $C^\infty$.
If $X$ and $Y$ are subsets of 
smooth
manifolds $M$ and $N$, we say a map $f\colon X\to Y$ is {\em smooth}
if $f$ can be extended to a smooth map from a \nbhd\ of $X$ in $M$
to a \nbhd\ of $Y$ in $N$.
We say $f$ is a {\em diffeomorphism} if it is smooth and has a smooth inverse.
We let $df$ denote the induced map on tangent bundles.

A {\em smooth collar} of a smooth manifold $M$ with boundary is a smooth embedding 
$c\colon \partial M\times [0,\infty)\to M$ so that $c(x,0)=x$ for all $x\in \partial M$.

A map is {\em proper} if the inverse image any compact set is compact.

If $f\colon X\to Y$ is a map and $A\subset X$, then
$f|_A\colon A\to Y$ is the restriction of $f$ to $A$
and $f|$ is the restriction of $f$ to some subset which is
clear from context, e.g., $f|\colon A\to Y$ implies that
$f| = f|_A$.

We let $\R$ denote the real numbers.
If $i=1,\ldots,n$ we let $\R^n_i$ denote the coordinate hyperplane
$\{x\in\R^n\mid x_i=0\}$.

If $\A$ is a tico then we call the codimension one submanifolds in $\A$ sheets.

An {\em affine subspace} of $\R^n$ is a translation of a 
linear subspace. An {\em affine map} is the composition of
a linear map and a translation.

The following corollary of Lemma \ref{division}
is useful so we state it here but prove it later.

\begin{lemma}\label{division1}
	Suppose $f\colon V\to \R$ is smooth where $V\subset \R^n$
	is open and $V\cap \R^n_i\subset f^{-1}(0)$.
	Then there is a smooth function $g\colon V\to \R$ so that
	$f(x)=x_ig(x)$ for all $x\in V$.
\end{lemma}

Much of this paper is quite elementary.
I find it generally easier to prove things than to search the literature and give
proper attribution to the original discoverers.
For this I apologize, and will be happy to include proper attributions if any readers
point them out to me.

\section{Tico Regular Neighborhoods}

Given a tico $\A$ in $X$ we'll define the notion of a tico regular \nbhd\ of $\A$.
It will be a mapping cylinder \nbhd, unique up to isotopy, etc., but in this version of the paper
we'll only develop the bare minimum we need which is not much.

If $\A$ is a tico in $X$, a \emph{tico parameterization} around $x\in |\A|$ is a smooth embedding
$h\colon U\to X$ where $U$ is an open \nbhd\ of $0$ in $\R^n$, $n=\dim X$, $h(0)=x$, and
$h^{-1}(|\A|) = U \cap \bigcup_{i=1}^k \R^n_i$ (recall $\R^n_i$ denotes $\{x\in \R^n\mid x_i=0\}$).

A \emph{tico map} $f\colon (X,|\A|)\to (\R,0)$ is a smooth map so that
$f^{-1}(0)=|\A|$ and for every $z\in |\A|$
and some (and hence every) tico parameterization $h\colon U\to X$ around $z$
with $h^{-1}(|\A|) = U \cap \bigcup_{i=1}^k \R^n_i$,
then $fh(x) = \Pi_{i=1}^k x_i^{b_i} u(x)$ for some nowhere 
zero smooth function $u$ and integer exponents $b_i>0$.
Note that after a change of coordinates $y_1 = x_1 |u(x)|^{1/b_1}$, $y_i=x_i$ for $i>1$,
we could choose our tico parameterization so that $u(x) = \pm1$ is constant.

\begin{lemma}\label{ticocrit}
	Let $\A$ be a tico in $X$ and let $f\colon (X,|\A|)\to (\R,0)$
	be a tico map.
	Then there is a \nbhd\ $U$ of $|\A|$ in $X$ so that
	$f$ has no critical points in $U-|\A|$.
	In particular, if $f$ is proper then there is an $\epsilon>0$
	so that $f$
	has no critical values in $[-\epsilon,0)\cup(0,\epsilon]$.
\end{lemma}

\begin{proof}
	To find $U$ we may as well suppose we are in the local case where 
	$X$ is an open subset of $ \R^{n}$,
	$|\A| =X\cap \bigcup_{i=1}^k \R^n_i$, and 
	$f(x) = \pm\Pi_{i=1}^k x_i^{b_i}$
	where $b_i>0$.
	In this case $0$ is the only critical value of $f$ so we may take $U=X$.
	
	If $f$ is proper, suppose $z_i\not\in |\A|$ is a sequence
	of critical points of $f$ so that $f(z_i)\to 0$.
	Then after taking a subsequence we may suppose $z_i\to z_0\in |\A|$ contradicting the existence of $U$.
\end{proof}

\begin{lemma}\label{istrn}
	Let $\A$ be a tico in $X$ with $|\A|$ compact.
	Then there is a proper tico map $f\colon (X,|\A|)\to (\R,0)$ and an $\epsilon>0$
	so that $f$ has no critical values in 
	$[-\epsilon,0)\cup (0,\epsilon]$.
\end{lemma}

\begin{proof}
	We may cover $|\A|$ with a finite number of tico parameterizations
	$p_i\colon V_i\to X$, $i=1,\ldots,n$
	where $V_i\subset  \R^{n}$
	and $p_i^{-1}(|\A|) = V_i\cap  \bigcup _{j=1}^{k_i}\R^m_j$.
	After shrinking each $V_i$ if necessary we may suppose
	that the closure of each $p_i(V_i)$ is compact.
	Let $U_0=X-|\A|$ and $U_i=p_i(V_i)$ for $i=1,\ldots,n$.
	Let $\alpha_i\colon X\to [0,1]$ be a smooth partition of unity
	for the open cover $U_0,\ldots,U_n$ of $X$.
	Let $g\colon X\to [1,\infty)$ be any proper smooth function.
	Let $g_i\colon V_i\to \R$ be $g_i(x) = \Pi_{j=1}^{k_i} x_j^2$.
	Then we may take 
	$$f(z) = \alpha_0(z)g(z) + \Sigma_{i=1}^n \alpha_i(z)g_ip_i^{-1}(z).$$
	Note $f$ is a tico map so by Lemma \ref{ticocrit}
	and properness of $f$ we may choose $\epsilon>0$ so that $f$ has no
	critical values in
	$[-\epsilon,0)\cup (0,\epsilon]$.
\end{proof}

If $\A$ is  a tico in $X$ with $|\A|$ compact 
a \emph{tico regular \nbhd} of $|\A|$ is $f^{-1}([-\epsilon,\epsilon])$
where $f\colon (X,|\A|)\to (\R,0)$ is a proper tico map and 
$f$  has no critical values in
$[-\epsilon,0)\cup (0,\epsilon]$
and $|f(x)|>\epsilon$ for all $x\in \partial X$.

If $\A=\{A_1,\ldots ,A_k\}$ is a tico in $X$ and $Y$ is a proper submanifold of $X$,
we say $Y$ is in general position with $\A$ if $Y$ is transverse to $\bigcap _{i\in B} A_i$
for all subsets $B\subset \{1,\ldots,k \}$.
In this case, there is an induced tico $Y\cap \A = \{Y\cap A_1,\ldots ,Y\cap A_k\}$ in $Y$.

If I were a good boy I would prove tico regular \nbhd s are mapping cylinder \nbhd s but instead content myself with the following weaker form which is all we need.

\begin{lemma}\label{tico_reg_nbhd}
	Suppose $\A$ is a tico in  $X$ with $|\A|$ compact and $U$ is a tico regular \nbhd\ of $|\A|$.
	Then $|\A|$ is a spine of $U$,  i.e., $U-|\A|$ is diffeomorphic to $\partial U\times [0,\infty)$.
	If $Y$ is a closed  submanifold of $X$ in general position with $\A$
	we may choose the regular \nbhd\ $U$ so that the pair $(U-|\A|,U\cap Y - |\A|)$ is diffeomorphic to
	$(\partial U,\partial U \cap Y)\times [0,\infty)$.
	In fact, for any proper tico map $f\colon (X,|\A|)\to (\R,0)$, and any sufficiently small $\epsilon>0$,
	$(U-|\A|,U\cap Y - |\A|)$ is diffeomorphic to
	$(\partial U,\partial U \cap Y)\times [0,\infty)$ for
	$U=f^{-1}([-\epsilon,\epsilon])$.
\end{lemma}

\begin{proof}
	Note $Y\cap \A$ is a tico in $Y$ so by Lemma \ref{ticocrit}  we may choose $\epsilon>0$ so that $f$ has no critical values in $[-\epsilon,0)\cup (0,\epsilon]$
	and $f|_Y$ has no critical values in $[-\epsilon,0)\cup (0,\epsilon]$ and $|f| > \epsilon$ on $\partial X$.
	We may choose a smooth vector field $v$ on $U-|\A|$ so that $df(v)=-f$ and $v$ is tangent to $Y$.
	It is easy to choose $v$ locally, then piece together with a partition of unity.
	Let $\phi_t$ be the flow for this vector field.
	Note $f\phi_t(x) = e^{-t}f(x)$.
	So $(x,t)\mapsto \phi_t(x)$ gives a diffeomorphism from $f^{-1}(\{\epsilon,-\epsilon\})\times [0,\infty)$
	to $f^{-1}([-\epsilon,0)\cup(0,\epsilon])$ which restricts to  a diffeomorphism from $f^{-1}(\{\epsilon,-\epsilon\})\cap Y\times [0,\infty)$
	to $f^{-1}([-\epsilon,0)\cup(0,\epsilon])\cap Y$.
	Since $U = f^{-1}([-\epsilon,\epsilon])$ and $|\A|=f^{-1}(0)$ the result follows.
\end{proof}

\section{Tico prespines}

The crucial step in our proof of Theorem \ref{main_thm} is the existence of what we will here call a tico prespine,
as suggested to us by Lowell Jones.
Suppose $\A$ is a tico in a smooth manifold $X$.
We say $\A$ is a \emph{tico prespine} if for a tico regular \nbhd\ $U$ of $|\A|$
then if $V$ is any connected component of $X-\Int U$ either
\begin{itemize}
	\item $V$ intersects $\partial X$ and $ V$ is diffeomorphic to $(V\cap \partial X)\times [0,1]$, or
	\item $V$ does not intersect $\partial X$ and $\partial V$ is diffeomorphic to a sphere, or
	\item $V$ does not intersect $\partial X$ and the component of $X-|\A|$ containing $V$ is diffeomorphic to $\R^n$.
\end{itemize}
The last alternative is only needed due to my imperfect knowledge of low dimensional topology.
If every manifold invertibly cobordant to $S^4$ is diffeomorphic to $S^4$ it is unneeded,
since it would be subsumed by the second alternative.

The proof of Theorem \ref{main_thm} then goes as follows.
Since the case $\dim X=1$ is trivial, we assume $\dim X \ge 2$.
We start with $M$ the boundary of $X$. If $X$ is not connected, take the connected sum of its components to make it so.
We then find a tico prespine $\A$ in $X$, which is what we detail in this paper.
Take a regular \nbhd\ $U$ of $|\A|$ as above. 
If there is a component $V$ of $X-\Int U$ which does not intersect $M=\partial X$ so that $\partial V$
is not diffeomorphic to a sphere then we delete the interior of a small closed disc in $V$ and add the result to $U$.
Thus we obtain a compact manifold $Y$ whose boundary is the disjoint union of $M$ and many spheres.
It will turn out that $|\A|$ is a spine of $Y$.

Consider the one complex $G$ whose vertices are the connected components of $X-|\A|$.
There is an edge between two components $V_1$ and $V_0$ if and only if there is a smooth path from $V_0$ to $V_1$
in general position with $|\A|$ and intersecting $|\A|$ in a single point.
Take a spanning forest $G'$ in $G$ so that each connected component of $G'$ contains exactly one vertex which is a 
component of $X-|\A|$ intersecting $\partial X$.  That is, $G'$ is a subcomplex of $G$ containing all the vertices of $G$
so that each connected component of $G'$ is simply connected and contains exactly one vertex which intersects $\partial X$.
(It's easy to construct $G'$ by starting with the vertices intersecting $\partial X$ 
and then growing $G'$ by adding one vertex and one edge at a time.)
Represent each edge $e$ of $G'$ by a smooth path $P_e$ between its two vertex components
in general position with $|\A|$ and intersecting $|\A|$ in a single point.
By shortening the path or isotoping it slightly if needed we may assume it is embedded and lies in $Y$ and all the paths $P_e$ are disjoint.
We constructed $Y$ using a tico regular \nbhd\ $U$ which is $f^{-1}([-\epsilon,\epsilon])$ for some proper tico map
$f$ and $\epsilon>0$. For convenience, after squaring $f$ we may assume $f$ is nonnegative.
By Lemma \ref{tico_reg_nbhd} we may choose an $\epsilon'\le\epsilon$
and a diffeomorphism $h\colon f^{-1}(\epsilon')\times [0,\infty)\to f^{-1}((0,\epsilon'])$
so that $$h((\bigcup_{e\in G'} P_e \cap f^{-1}(\epsilon'))\times [0,\infty))= \bigcup_{e\in G'} P_e \cap f^{-1}((0,\epsilon']).$$
Since $f$ has no critical values in $[\epsilon',\epsilon]$ we know $U-f^{-1}([0,\epsilon'))$ is just the product with an interval,
so we may as well have taken $U=f^{-1}([0,\epsilon'])$ in our construction of $Y$.
Assume we did so, so from now on $\epsilon'=\epsilon$.

Suppose we have some component $V$ of $X-\Int U$ which doesn't intersect $\partial X$ and whose boundary is not a sphere.
Then $V$ is contained in a component $V'$ of $X-|\A|$ diffeomorphic to $\R^n$.
Since the end if $V'$ is trivialized by both $S^{n-1}\times \R$ and $(V'\cap f^{-1}(\epsilon))\times \R$
we know $S^{n-1}$ and $V'\cap f^{-1}(\epsilon)=\partial V$ are invertibly cobordant, hence homotopy equivalent.
Thus $n>3$ (or in fact $n=5$ after a great deal of work, but let's not bother).
In particular, we know by \cite{CKS} that the curves $P_e\cap V'$ are all standard, i.e.,
there is a diffeomorphism $g\colon V'\to \R^n$ so that for each edge $e$, either $P_e\cap V'$ is empty or
is the radial ray
$g(P_e\cap V') = \{tv_e\mid t\ge 1\}$ for some nonzero vector $v_e\in \R^n$.
We may as well extend or truncate the paths $P_e$ so that all $v_e$ are unit vectors.
Then when we form $Y$ we just make sure that $Y\cap V' = g^{-1}(\{x\in \R^n \mid 1\le |x|\})$.

After all this we have arrived at the point where $(Y-|\A|, \bigcup_{e\in G'} P_e\cap Y-|\A|)$ is diffeomorphic with 
$(\partial Y,\partial Y\cap \bigcup_{e\in G'} P_e)\times [0,\infty)$.
This diffeomorphism can be obtained by integrating a vector field $v$ on $Y-|\A|$.
Now for each edge $e$ we attach a one handle $[-1,1]\times D^{n-1}$ to $Y$ at the two endpoints of $P_e$,
obtaining a manifold $Y'$.
We have a curve $C_e = P_e \cup [-1,1]\times 0$ obtained from $P_e$ by attaching the core of the handle.
Clearly $|\A|\cup \bigcup_{e\in G'}C_e$ is a spine of $Y'$, we can see this by integrating the vector field $v'$
on $Y-|\A|-\bigcup_{e\in G'}C_e$ given by $v$ on $Y$ and $v'(x,y)= (x,-y)$ on each one handle.
Let $\pi\colon Z\to Y'$ be the blowup with center $\bigcup_{e\in G'} C_e$.
Let $\pi^*(\A)$ denote the set of strict transforms of of the sheets of $\A$.
Then $\A' = \pi^*(\A)\cup \{\pi^{-1}(C_e)\}_{e\in G'}$ is a tico in $Z$ and $|\A'|$ is a spine of $Z$
and $\partial Z = M$.

\section{Cellular Ticos}

Our first step to a tico prespine will be a cellular tico.
We say a tico $\A$ in $X$ is \emph{cellular} if 
for every sheet $S\in \A$ there is a codimension 0 submanifold $B_S$ of $\Int X$ 
so that $S=\partial B_S$ and
every connected component
of $X-|\A|$ is either diffeomorphic to $\R^n$ or not contained in any $B_S$.

\begin{theorem}\label{cellular_tico}
	Suppose $\A$ is a tico in a smooth $n$ dimensional manifold $X$ and for each sheet $S$ of $\A$ there is a
	codimension 0 submanifold $B_S$ of $X$ so that $S=\partial B_S$ and
	$B_S-S$ is diffeomorphic to $\R^n$.
	Moreover, suppose that for any $C\subset \A$ with $\bigcap _{S\in C} B_{S}$ nonempty then $\bigcap _{S\in C} S$
	is diffeomorphic to a (nonempty) sphere.
	Then any connected component of $X-|\A|$ which is contained in some $B_S$ is diffeomorphic to $\R^n$.
	In particular $\A$ is cellular.
\end{theorem}

\begin{proof}
	Choose some sheet $T\in \A$.
	By hypothesis we know $T$ is diffeomorphic to the sphere $S^{n-1}$.
	Let $\A' = \A-\{T\}$.
	By induction on the number of sheets of $\A$, we know that any connected component of $X-|\A'|$ contained in 
	some $B_S$, $S\neq T$, is diffeomorphic to $\R^n$.
	Consider the tico $\A'\cap T$ in $T$ whose sheets are all nonempty $S\cap T$ for $S\in \A'$.
	If $B_S\cap T$ is nonempty for some sheet $S\ne T$ then $S\cap T$ is an $n-2$ sphere so
	by the weak Schoenflies Theorem \ref{schoen} we know that $B_S\cap T - S\cap T$ is diffeomorphic to $\R^{n-1}$.
	Hence by induction on the number of sheets we know that any connected component of $T-|\A'\cap T|$ contained in 
	some $B_S\cap T$, $S\neq T$, is diffeomorphic to $\R^{n-1}$.
	So by Theorem \ref{HLT} we know any connected component of $X-|\A|$ which is contained in some $B_S$, $S\ne T$ is diffeomorphic to $\R^n$.
	Repeating the argument with some $T'\ne T$ in place of $T$ we see that any connected component of $X-|\A|$ which is contained in $B_T$ is also diffeomorphic to $\R^n$.
\end{proof}

We used above the following weak version of the smooth Schoenflies theorem since as far as I know, smooth Schoenflies is unknown in dimension 4.

\begin{theorem}\label{schoen}
	Suppose $f\colon S^{n-1}\to S^n$ is a smooth embedding.  Then $S^n-f(S^{n-1})$ is diffeomorphic to the disjoint union of two copies of $\R^n$.
\end{theorem}

\begin{proof}
	The proof is by infinite swindle\footnote{The reader has probably seen an (invalid) infinite swindle joke proof that
		$0=1$, i.e., $0 = (1-1)+(1-1)+ \cdots = 1 +(-1+1)+(-1+1)+\cdots = 1 + 0 = 1$.
		Amazingly, this sort of proof is correct in a number of topological instances, three that I know of,
		and gives short magic proofs of powerful results.
		This paper uses all three instances I know of: Schoenflies, hyperplane linearization, and invertible 
		cobordisms.
		}. Let $A$ and $B$ be the closures of the two components of $S^n-f(S^{n-1})$.
	Let $A'$ and $B'$ be obtained from $A$ and $B$ by gluing a disc to the boundary.
	Note $A'\#B'  = S^n$.
	Consider the infinite connected sum $\R^n\#A'\#B'\# \cdots$.
	On the one hand, it is $\R^n\#(A'\#B')\# (A'\#B')\#\cdots = \R^n$.
	On the other hand it is $\R^n\#A'\#(B'\# A')\#(B'\# A')\#\cdots = \R^n\#A' = \Int A$.
	
	I'll be more precise about the infinite connected sum above.
	For $i=1,2,\ldots$, let $D_i\subset \R^n$ be a closed disc of radius $1/4$ around some point of distance $i+1/2$
	from the origin.
	Let $X = \R^n - \bigcup _{i=1}^\infty \Int D_i$. 
	Then we obtain the infinite connected sum by gluing a copy of $A$ to $\partial D_i$ for $i$ odd
	and gluing a copy of $B$ to $\partial D_i$ for $i$ even.
\end{proof}


\begin{theorem}\label{HLT}
	Suppose $\A$ is a tico in a smooth manifold $X$ and $Z$ is a connected component of $X-|\A|$ diffeomorphic to $\R^n$.
	We suppose $Z$ has compact closure.
	Let $N$ be a proper codimension one submanifold of $X$ in general position with $\A$ so that
	$N\cap Z$ is diffeomorphic to the disjoint union of $k$ copies of $\R^{n-1}$.
	Then $Z-N$ is diffeomorphic to the disjoint union of $k+1$ copies of $\R^n$.
	If $k=1$, the pair $(Z,N)$ is diffeomorphic to $\R^{n-1}\times (\R,0)$.
\end{theorem}

\begin{proof}
	If $n\ne 3$ this is immediate from the hyperplane linearization theorem in \cite[Theorem 7.3]{CKS}.
	So suppose $n=3$. 
	By Lemma \ref{tico_reg_nbhd} there is a tico regular \nbhd\ $U$ of $|\A|$
	and a diffeomorphism $h\colon \partial U\times [0,\infty) \to U-|\A|$
	so that $h^{-1}(N) = (\partial U\cap N) \times [0,\infty) $.
	Since $h$ trivializes the end of $Z\approx \R^3$ we must have $\partial U \cap Z$ diffeomorphic to the two sphere $S^2$.
	By smooth Schoenflies, $Z-\Int U$ is diffeomorphic to the 3 ball, so we have a diffeomorphism
	$g\colon \R^3 \to Z$ so that $g(x) \in U$ if and only if $|x|\ge 1$ and if $|x|\ge 1$ then $g(x) = h(g(x/|x|,|x|-1))$.
	The result now follows from \cite[Theorem 6.1]{CKS} since for any $x\in \partial U \cap N$, 
	the ray $h(x\times [0,\infty))$ is  unknotted in $N$ and $Z$.
\end{proof}

\section{Convex Sets}

The following allows us to take a piecewise linear convex set and perturb it to be smooth.

\begin{lemma}\label{convex_smooth}
	Suppose $l_i$, $i=1,\ldots, k$ is a collection of degree 1 polynomials in $n$ variables and 
	$L\subset \R^n$ is an affine subset.
	Trivially, 
	$$X=\{x\in L \mid l_i(x)\ge 0 \ \mathit{for}\ i=1,\ldots,k\}$$
	 is a convex set although  probably not a smooth manifold with boundary.
	For any $\epsilon>0$ consider the set $X_\epsilon = \{x\in X \mid \prod_{i=1}^k l_i(x)\ge \epsilon\}$.
	Then $X_\epsilon $ is convex and for small enough $\epsilon$, $X_\epsilon$ is a smooth submanifold of $L$.
\end{lemma}

\begin{proof}
	By Sard's theorem, the critical values of $f(x) = \prod_{i=1}^k l_i(x)$ have measure 0.
	But the critical values of $f$ are also a semialgebraic set (being the polynomial image of the algebraic set of critical points),
	and hence must be a finite set. Consequently, small enough $\epsilon >0$ are regular values of $f$
	and hence $X_\epsilon$, being $X\cap f^{-1}([\epsilon,\infty))$ is a smooth submanifold.
	
	The convexity of $X_\epsilon$ is a simple calculus exercise.
	Suppose $x,y\in X_\epsilon$ are distinct points.  I claim that all critical points of $h(t) = f(tx+(1-t)y)$ for $t\in (0,1)$
	are local maxima.
	Consequently, $tx+(1-t)y\in X_\epsilon$ for all $t\in [0,1]$ and convexity is established.
	
	To see the calculus claim, note $l_i(tx+(1-t)y)= a_it+b_i$ for some $a_i$ and $b_i$.
	We may as well assume some $a_i$ is nonzero, otherwise $h$ is constant.
	Then $h'(t) = h(t) \Sigma_{i=1}^k a_i/(a_it+b_i)$ and $h''(t) = h(t)((\Sigma_{i=1}^k a_i/(a_it+b_i))^2 - \Sigma_{i=1}^k a_i^2/(a_it+b_i)^2)$.
	So if $h'(t)=0$ then $h''(t) = -h(t)\Sigma_{i=1}^k a_i^2/(a_it+b_i)^2 < 0$.
	So by the second derivative test we have a local maximum.
\end{proof}

\section{Distancelike Functions}

Suppose $M$ is a smooth manifold,  $N$ is a smooth submanifold of $M$, and
$\rho\colon U\to [0, \infty)$ is a smooth function from an
open \nbhd\ $U$ of $N$ in $M$.
We say that $\rho$ is {\em \MB}\footnote{I originally called these Morse-Bott functions but Wikipedia tells me that name is already in
	use for a different setup, more general in that you don't have a local minimum.
	So I'll call them \MB\ functions instead.}
  around $N$ if:
\begin{itemize}
	\item At every point of $N$
	the Hessian of $\rho$ has rank equal to the codimension of $N$.
	\item $N\subset \rho^{-1}(0)$.
\end{itemize}
The following Lemma gives a local description of a \MB\ function as well as other Taylor's theorem type results.

\begin{lemma}\label{division}
	Let $V$ be an open subset of $N\times \R^m$
	for some smooth manifold $N$.
	Suppose $f\colon V\to \R$ is a smooth function
	which vanishes on $N\times 0$. Then:
	\begin{enumerate}
		\item There are smooth functions $g_i\colon V\to \R$
		so that $f(x,y) = \Sigma _{i=1}^m y_ig_i(x,y)$.
		\item If  $f$ is nonnegative,
		there are smooth functions $h_{ij}\colon V\to \R$ so that
		$f(x,y) = \Sigma_{i=1}^m\Sigma_{j=1}^m y_iy_j h_{ij}(x,y)$.
		In other words, there is a smooth  $m\times m$ symmetric matrix valued function $L$ on $V$
		so that $f(x,y) = y^TL(x,y)y$.
		\item If $f$ is \MB\ around $N\times 0$
		then:
		\begin{enumerate}
			\item $N\times 0$ is an open subset of $f^{-1}(0)$.
			\item If $L$ is as above then for each $x\in N$ the
			quadratic form
			$y\mapsto y^TL(x,0)y$
			on $\R^m$ is positive definite.
		\end{enumerate}
	\end{enumerate}
\end{lemma}

\begin{proof}
	To prove 1, 
	it suffices to find the $g_i$ locally and
	piece together with a smooth partition of unity.
	Locally away from $N\times 0$ if $y_i\ne 0$ we may take
	$g_i(x,y) = f(x,y)/y_i$ and $g_j=0$ for $j\ne i$.
	Near a point of $N\times 0$, we may take
	$g_i(x,y) = \int_0^1 \partial f/\partial y_i(x,ty)\,dt$.
	Note that 
	\begin{eqnarray*}
		\Sigma _{i=1}^m y_ig_i(x,y) &=& \Sigma _{i=1}^m y_i\int_0^1 \partial f/\partial y_i(x,ty)\,dt\\
		&=& \int_0^1 \Sigma _{i=1}^m y_i\partial f/\partial y_i(x,ty)\,dt\\
		&=& \int_0^1  d/dt ( f(x,ty))\,dt\\
		&=& f(x,y) - f(x,0) = f(x,y).
	\end{eqnarray*}
	
	To prove 2, note that each $g_i$ must vanish on $N\times 0$,
	otherwise if we take $(x,y)$ near $N\times 0$ with $y_j=0$ for $ j\neq i$
	then $f(x,y) = y_ig_i(x,y)$ takes on negative values.
	So applying 1 to each $g_i$ we get
	$g_i(x,y) = \Sigma_{j=1}^m y_jh_{ij}(x,y)$ and 2 follows.
	We get $L$ by setting the $i,j$-th entry of $L(x,y)$ to $(h_{ij}(x,y)+h_{ji}(x,y))/2$.
	
	To prove 3, we may as well suppose we are in the local case $N=\R^n$.
	Note that 
	$\partial^2 f/\partial x_i\partial x_j(x,0) = 0$ and $\partial^2 f/\partial x_i\partial y_j(x,0) = 0$
	since $f$ vanishes on $N\times 0$.
	On the other hand 
	\begin{eqnarray*}
		\partial f/\partial y_i(x,y) &=&
		\Sigma_{j=1 }^m y_j(h_{ij}(x,y) +h_{ji}(x,y) + \Sigma_{k=1}^m y_kg_{ijk}(x,y))\\
		\partial^2 f/\partial y_i\partial y_j(x,y) &=&
		h_{ij}(x,y) +h_{ji}(x,y)+ \Sigma_{k=1}^m y_k g'_{ijk}(x,y)
	\end{eqnarray*}
	for some smooth functions $g_{ijk}$ and $g'_{ijk}$.
	Thus $\partial^2 f/\partial x_i\partial x_j(x,0) = h_{ij}(x,0)+h_{ji}(x,0)$.
	So the Hessian of $f$ is 
	$\begin{bmatrix}
	0&0\\
	0&2L(x,0)\\ 
	\end{bmatrix}$ at $(x,0)\in N\times 0$.  Since the Hessian has rank $m$
	on $N\times 0$ we know $L(x,0)$ is nonsingular,
	and hence the quadratic form $y\mapsto y^TL(x,0)y$ is nondegenerate.
	If this quadratic form is not positive definite there
	is a $y$ so that $y^TL(x,0)y<0$.
	But then for small enough $t$ we have
	$$f(x,ty) = t^2 y^TL(x,ty)y \approx t^2 y^TL(x,0)y <0$$
	contradicting the nonnegativity of $f$.
	So the quadratic form $y\mapsto y^TL(x,0)y$ is positive definite.
	Note the quadratic form must then by continuity be positive definite for all $(x,y)$ near $N\times 0$ which means $f(x,y)>0$
	for $(x,y)$ near $N\times 0$ but not on $N\times 0$.
	Consequently $N\times 0$ is an open subset of $f^{-1}(0)$.
\end{proof}

\begin{lemma}\label{paramMorseBott}
	Suppose $\rho\colon U\to [0,\infty)$ is \MB\ around $N$ and $\pi\colon U\to N$ is a smooth retraction.
	Moreover, suppose the normal bundle of $N$ is trivial.
	Then there is a smooth embedding $h\colon V\to U$ onto a \nbhd\ of $N$ so that:
	\begin{enumerate}
		\item $V$ is a \nbhd\ of $N\times 0$ in $N\times \R^k$.
		\item $\pi h(x,y) = h(x,0)$ for all $(x,y)\in V$.
		\item $\rho h(x,y) = |y|^2$ for all $(x,y)\in V$. 
	\end{enumerate}
\end{lemma}

\begin{proof}
	This is just a parameterized Morse Lemma.
	By induction we may assume there is a smooth embedding $h'\colon V'\to U$ onto a \nbhd\ of $N$ so that:
	\begin{enumerate}
		\item $V'$ is a \nbhd\ of $N\times 0$ in $N\times \R^k$.
		\item $\pi h'(x,y) = h(x,0)$ for all $(x,y)\in V'$.
		\item $\rho h'(x,y) = \Sigma _{i=1}^{\ell-1} y_i^2 + \Sigma_{i=\ell}^k\Sigma_{j=\ell}^k y_iy_j g_{ij}(x,y)$ for some smooth functions $g_{ij}$ and all $(x,y)\in V$. 
	\end{enumerate}
	Note that by positive definiteness, $g_{\ell\ell}(x,0) > 0$ for all $x\in N$, so we may as well shrink $V'$
	so that $g_{\ell\ell}(x,y) > 0$ for all $(x,y)\in V'$.
	We also may as well suppose that $g_{ij}=g_{ji}$.
	Define $f\colon V'\to N\times \R^k$ by $f(x,y) = (x,y')$ where $y'_i=y_i$ for $i\ne \ell$ and
	$y'_\ell = \sqrt{g_{\ell\ell}(x,y)}y_\ell + \Sigma_{i=\ell+1}^k y_i g_{\ell i}(x,y)/\sqrt{g_{\ell\ell}(x,y)}$.
	By the inverse function theorem we may, after shrinking $V'$, assume that $f$ gives a diffeomorphism
	from $V'$ to a \nbhd\ $V''$ of $N\times 0$.
	Define $h''\colon V''\to U$ by $h'' = h'f^{-1}$.
	Then $\rho h''(x,y) = \Sigma _{i=1}^{\ell} y_i^2 + \Sigma_{i=\ell+1}^k\Sigma_{j=\ell+1}^k y_iy_j g''_{ij}(x,y)$ for some smooth functions $g''_{ij}$, so we are done by induction.
\end{proof}

\section{Simplicial Stratified Sets}

The standard $n$ simplex $\Delta^n$ is $\{x\in \R^{n+1}\mid x_i\ge 0,\ 
\Sigma_{i=0}^n x_i= 1\}$.
It has $n+1$ vertices $e_0,\ldots,e_n$ where $e_i$ has a 1 in the $i$-th cooordinate and all other coordinates 0.
Any face of the simplex corresponds to a nonempty subset $A\subset \{0,\ldots,n \}$ where the face is the convex hull
of $\{e_i\mid i\in A\}$.  This face is also $\mu_A^{-1}(0)$ where $\mu_A\colon \Delta^n\to \R$ is the linear function
$\mu_A(x) = \Sigma _{i\not\in A} x_i$.
An open face of the simplex will be those points in a face which are not in any lower dimensional face.
It is defined by $\mu_A(x) = 0$ and $x_i>0$ for $i\in A$.

A stratified subset $Z$ of a smooth manifold $X$ is 
a collection of disjoint smooth submanifolds called strata.
We ask that if $S$ and $T$ are strata and $S\cap \Cl T$ is nonempty, then $S\subset \Cl T$.
The notation $S\prec T$ will mean that $S\ne T$ and $S\subset \Cl T$.
We also require local finiteness, any $x\in X$ has a \nbhd\ which intersects only finitely many strata.
We also ask that the union of the strata be a closed set.
In this paper we specify that all strata have empty boundary and no stratum intersects the boundary of $X$,
although outside this paper it is often useful to allow this.

We say a stratified subset $Z$ of $X$ is 
\emph{simplical}
if for each stratum $S$ there is a diffeomorphism 
$f\colon \Delta^k\to \Cl S$ so that for each open face $\sigma$
of $\Delta^k$, $f(\sigma)$ is a stratum of $Z$.
(We call $f$ a parameterization of the stratum $S$.)
One example of a simplicial stratified set is a smooth triangulation.

Whitney invented some conditions A and B on stratified sets which make them tractable to work with,
see \cite{GWPL} for example.
A simplicial stratified set must satisfy the
Whitney conditions,  this is because the open faces of a simplex
trivially satisfy the Whitney conditions.

Since a simplicial stratified set satisfies the Whitney conditions,
we have Thom's tubular data which we shall use to construct a tico.
We find it easier at one point to require a bit more from this tubular data
than is usual.
In particular, the tubular data includes for each stratum $S$ a function $\rho_S$
which is \MB\ around $S$,
as well as a smooth local retraction $\pi_S$ to the stratum $S$.
We find it useful to be able to specify the functions $\rho_S$ in advance, and then
afterwards construct the rest of the tubular data $\pi_S$ satisfying the compatibility conditions.
Perhaps this result is in the literature somewhere, but in any case I prove it in an appendix.

We can now state our strengthened result on existence of tubular data,
which follows immediately from Lemma \ref{piconstruct} and Lemma \ref{paramMorseBott}.

\begin{theorem}\label{piconstruct0}
	Let $X$ be a simplicial stratified subset of a smooth manifold $M$
	and suppose that for each stratum $S$ of $X$ we choose a
	\MB\
	function $\rho_S\colon U_S\to [0,\infty)$
	around  $S$. 
	Then, after perhaps  shrinking the \nbhd s $U_S$, there
	are smooth retractions $\pi_S\colon U_S\to S$ 
	and smooth embeddings $\nu_S\colon U_S \to S\times \R^{\dim X-\dim S}$
	onto a \nbhd\ $V_S$ of $S\times 0$
	so that
	\begin{enumerate}
		\item If $T\prec S$ then $U_S\cap U_T= \pi_S^{-1}( S\cap U_T)$.
		\item If $T\prec S$ then $\pi_T\pi_S = \pi_T$ and
		$\rho_T\pi_S = \rho_T$ on $U_S\cap U_T$.
		\item If $U_S\cap U_T$ is nonempty, then either $S\prec T$,
		$T\prec S$, or $S=T$.
		\item If $T_1\prec T_2\prec\cdots\prec T_\ell\prec S$ then
		$$(\rho_{T_1},\rho_{T_2},\ldots,\rho_{T_\ell},\pi_{T_1})
		\colon S\cap U_{T_1}\cap\cdots\cap U_{T_\ell} \to
		(0,\infty)^\ell\times T_1$$ is a submersion.
		\item $\nu_S\pi_S\nu_S^{-1}(x,y) = (x,0)$ for all $(x,y)\in V_S$.
		\item $\rho_S\nu_S^{-1}(x,y) = |y|^2$ for all $(x,y)\in V_S$.
		\item There are smooth $\gamma_S\colon S\to (0,\infty)$ so that 
		$V_S = \{(x,y)\in S\times \R^{\dim X-\dim S}\mid \gamma_S(x)>|y|^2  \}$.
	\end{enumerate}
\end{theorem}

It is easy to construct a \MB\ function for a submanifold $N$, especially if its
normal bundle is trivial as is the case for a stratum of a simplicial stratified set.
Take a \nbhd\ $U$ of $N$ diffeomorphic to $N\times \R^k$ and take norm squared of the $\R^k$ factor.

\section{Constructing a Cellular Tico}

Suppose $Z$ is a simplicial stratified subset of a smooth manifold $X$
and $\rho_S$, $\pi_S$, and $\nu_S$ are tubular data statisfying the conclusions of Theorem \ref{piconstruct0}.
We will construct a cellular tico in $X$ with one sheet $A_S$ for each stratum $S$ of $Z$.
This sheet $A_S$ will be the boundary of a codimension 0 submanifold $B_S$ of $X$ diffeomorphic to a disc.
The construction will depend on some parameters $\tau_S>0$ and $\tau_{S,i}$ for $\dim S \le i \le \dim Z$.
They must satisfy
\begin{itemize}
	\item $0 < \tau_{S,\dim S} < \tau_{S,\dim S+1} <\cdots <\tau_{S,\dim Z}<\tau_S$ and
	\item if $K_S = S-\bigcup_{T\prec S}\rho_T^{-1}((0,\tau_{T,\dim T}))$, $x\in K_S$, $y\in \R^{\dim X-\dim S}$, and $|y|^2\le \tau_S$
	then $(x,y)\in \nu_S(U_S)$.
	In particular, $\pi_S^{-1}(K_S) \cap \rho_S^{-1}([0,\tau_S])$ is compact
	and diffeomorphic to $K_S\times$ a disc.
\end{itemize}
We may as well standardize
$\tau_{S,i}$, say $\tau_{S,i} = \tau_S(1+i)/(2+\dim Z)$.

\begin{lemma}
	There are choices of $\tau_S$ satisfying the required conditions.
\end{lemma}

\begin{proof}
	We may suppose by induction that we have chosen $\tau_S$ for all $S$ with dimension $<k$
	satisfying the above conditions. Now pick a stratum $S$ of dimension $k$.
	We can define $K_S$ since we have already chosen $\tau_T$ for $T\prec S$.
	Note that $K_S$ is compact since it is the closed subset
	$K_S =  \Cl S-\bigcup_{T\prec S}\rho_T^{-1}((-\infty,\tau_{T,\dim T}))$
	of $\Cl S$.
	Now choose $\tau_S < \min_{x\in K_S}\gamma_S(x)$.
\end{proof}

We start with a piecewise smooth approximation $J_S$ of $B_S\cap S$, 
$$J_S = S - \bigcup_{T\prec S} \rho_T^{-1}((0,\tau_{T,\dim S})).$$

Lemma  \ref{simplex_param} below shows that we may find a parameterization $h\colon \Delta^k\to \Cl S$ of $S$ so that $h^{-1}(J_S)$
is the intersection of linear half spaces.
It probably  doesn't matter too much how we smooth $J_S$, but to be precise we will use Lemma \ref{convex_smooth}
to smooth $J_S$, obtaining a subset $E_S$ of $J_S$ with smooth boundary so that $h^{-1}(E_S)$ is convex.
We will take the smoothing $E_S$ close enough to $J_S$ that 
$$\partial E_S \subset \bigcup_{T\prec S} \rho_T^{-1}((0,\tau_{T,\dim S+1})).$$
By Lemmas \ref{convex_smooth} and \ref{simplex_param} we may also suppose that for every collection of faces $T_1\prec\cdots\prec T_m\prec S$,
$h^{-1}(E_S\cap \bigcap_{i=1}^m \rho_{T_i}^{-1}(\tau_{T_i})) $ is convex with smooth boundary.

We now let $B_S$ be a smoothing of $\pi_s^{-1}(E_S)\cap \rho_S^{-1}([0,\tau_S])$.
Note that $\pi_s^{-1}(E_S)\cap \rho_S^{-1}([0,\tau_S])$ is diffeomorphic via $\nu_S$ to  a product of smooth discs
$E_S\times D^{n-\dim S}$ so we only need smooth the corner $\pi_S^{-1}(\partial E_S)\cap \rho_S^{-1}(\tau_S)$ in some standard way.
We will smooth it by taking a continuous  $\beta_S\colon  E_S\to [0,\tau_S]$ so that $\beta_S(x)=\tau_S$ if $x\not\in 
 \bigcup_{T\prec S} \rho_T^{-1}((0,\tau_{T,\dim S+1}))$, $\beta|_{E_S-\partial E_S}$ is smooth, 
$\beta(x)\ge \tau_{S,i}$ for all $i$, and $\{(x,t)\in E_S\times \R\mid t\le \beta_S(x)\}$  is a smooth submanifold of $S\times \R$.
We may then let $$B_S
= \nu_S^{-1}(\{(x,y)\in E_S\times \R^{n-\dim S}\mid \beta_S(x) \ge |y|^2 \}).$$

Figure 2 shows $B_T$ for a 2 dimensional $T$ and all its faces in a 2 dimensional manifold. 
Figure 3 shows $h^{-1}(B_T)$ for the good parameterization $h$ which linearizes the $\rho_T$.

\begin{figure}
	\centering
	\includegraphics[scale=0.4]{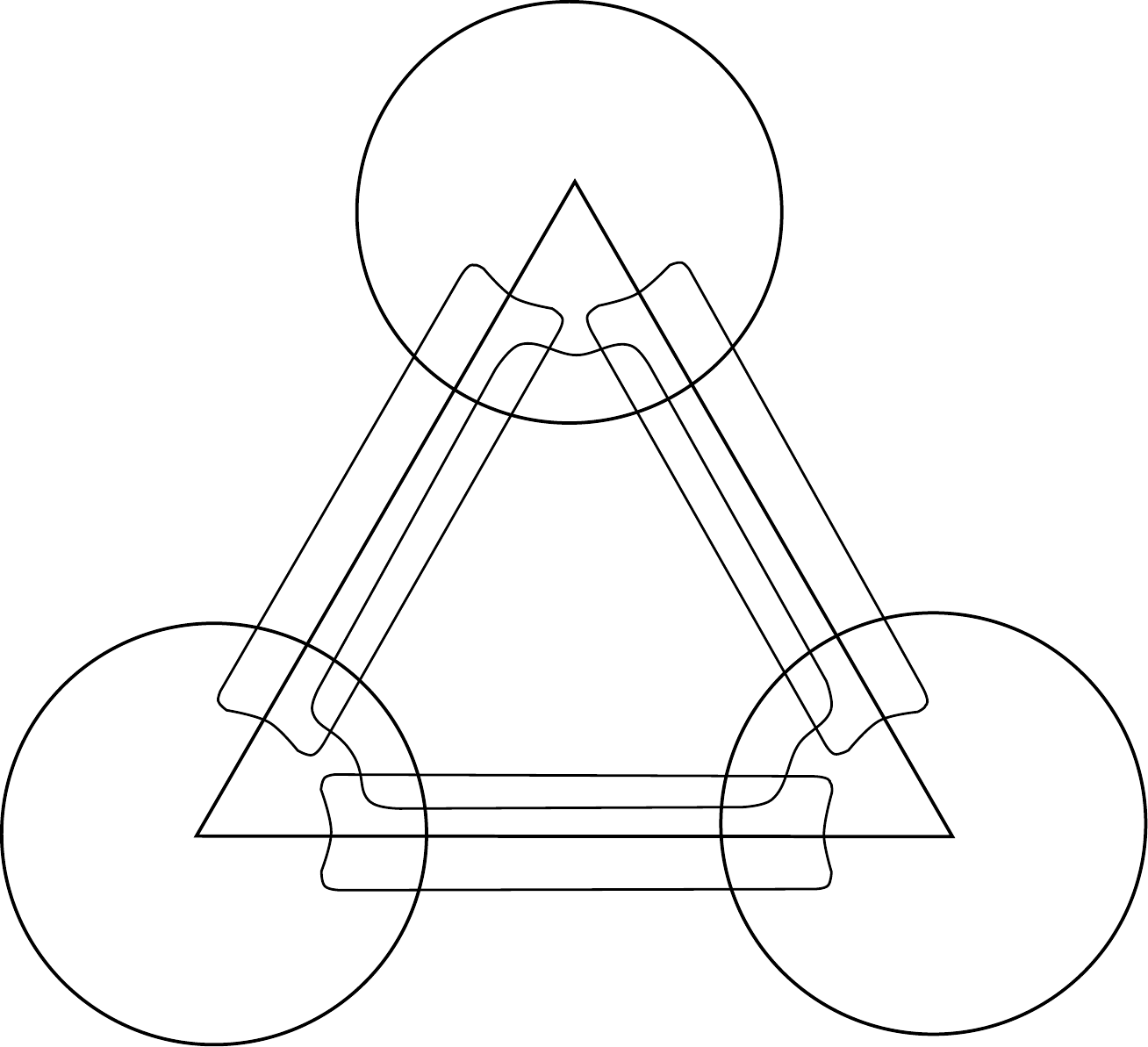}
	\caption{$B_T$ for $\Delta^2$ and its faces in a 2 manifold}
\end{figure}

\begin{figure}
	\centering
	\includegraphics[scale=0.4]{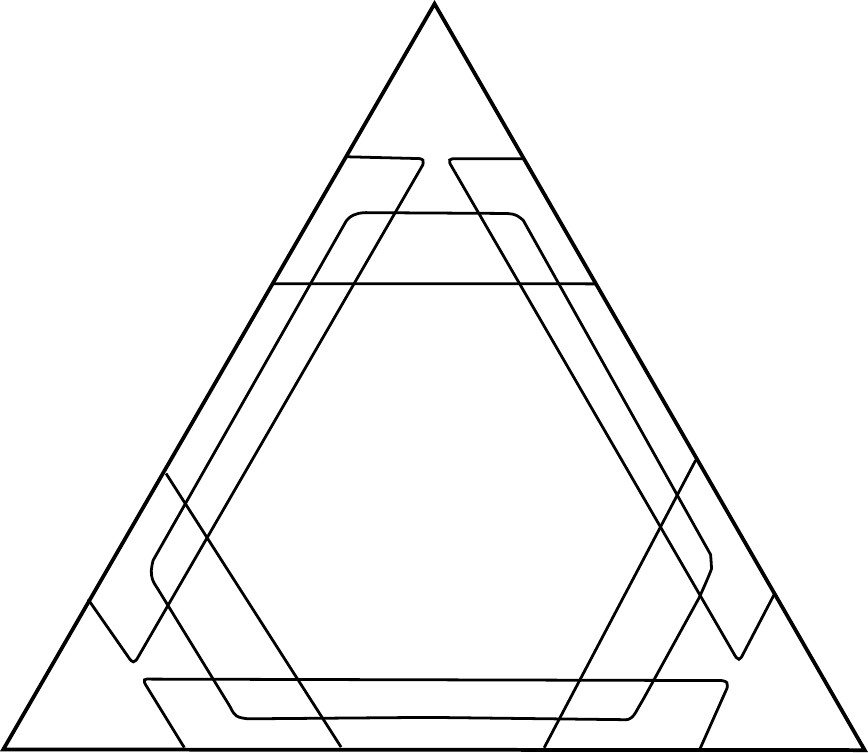}
	\caption{$h^{-1}(B_T)$ for $\Delta^2$ and its faces}
\end{figure}

\begin{lemma}
	The above tico satisfies the hypotheses of Theorem \ref{cellular_tico} and hence the tico is cellular.
\end{lemma}

\begin{proof}
	Clearly $B_S-A_S$ is  diffeomorphic to $\R^n$, since $E_S = B_S\cap S$ is a convex smooth manifold,
	hence a disc, and $B_S$ is  the product of $E_S$ with another disc, with corners rounded.
	Now suppose we have  $B_{T_1},\ldots,B_{T_k}$ with nonempty intersection.
	After reordering, we can have $T_1\prec T_2\prec \cdots\prec T_k$.
	Set $S=T_k$.
	We will show that $B_S\cap \bigcap_{i=1}^{k-1} A_{T_i}$ is a disc, and hence its boundary
	$A_S\cap \bigcap_{i=1}^{k-1} A_{T_i}$ is a sphere, at which point we are done.
	
	Suppose $x\in B_S\cap A_{T_i}$ for some $i<k$.
	We know $\pi_{T_i}\pi_S(x) = \pi_{T_i}(x) \in E_{T_i}$.
	Also $\rho_{T_i}\pi_S(x) = \rho_{T_i}(x) \le \tau_{T_i}$.
	Suppose $\rho_{T_i}\pi_S(x) < \tau_{T_i}$.  
	Then $\pi_{T_i}(x)$ is near the boundary of $E_{T_i}$, in particular there is a $T\prec T_i$
	so that $\tau_{T,\dim T_i}\le \rho_T\pi_{T_i}(x) < \tau_{T,\dim T_i+1}$.
	But  $\pi_S(x)\in E_S$ so 
	$$\rho_T\pi_S(x)\ge \tau_{T,\dim S} \ge  \tau_{T,\dim T_i+1} > \rho_T\pi_{T_i}(x) = \rho_T(x) = \rho_T\pi_S(x),$$
	a contradiction.
	Consequently $\rho_{T_i}(x) = \rho_{T_i}\pi_S(x) = \tau_{T_i}$.
	So $B_S\cap A_{T_i}\subset B_S\cap \rho_{T_i}^{-1}(\tau_{T_i})$.
	
	Conversely, suppose $x\in B_S\cap \rho_{T_i}^{-1}(\tau_{T_i})$.
	For any $T\prec T_i$ we know either $\pi_S(x)\not\in U_T$ or $\rho_T\pi_S(x)\ge \tau_{T,\dim S}\ge \tau_{T,\dim T_i+1}$.
	Consequently, $\pi_{T_i}(x)$ is far enough from the boundary of $E_{T_i}$ that $x\in A_{T_i}$.
	
	So $B_S\cap A_{T_i}= B_S\cap \rho_{T_i}^{-1}(\tau_{T_i})$ and consequently,
	$$B_S\cap \bigcap_{i=1}^{k-1} A_{T_i} = B_S \cap \bigcap _{i=1}^{k-1} \rho_{T_i}^{-1}(\tau_{T_i}).$$
	Note $\pi_S(B_S \cap \bigcap _{i=1}^{k-1} \rho_{T_i}^{-1}(\tau_{T_i}))
	= E_S \cap \bigcap _{i=1}^{k-1} \rho_{T_i}^{-1}(\tau_{T_i})$ and so
	$$\nu_S(B_S \cap \bigcap _{i=1}^{k-1} \rho_{T_i}^{-1}(\tau_{T_i}))
	= \{(x,y)\in (E_S\cap \bigcap _{i=1}^{k-1} \rho_{T_i}^{-1}(\tau_{T_i}))\times \R^{n-\dim S}
	\mid \beta_S(x) \ge |y|^2 \}$$
	which is just $E_S\cap \bigcap _{i=1}^{k-1} \rho_{T_i}^{-1}(\tau_{T_i})$ cross a disc with corners rounded.
	But we know $E_S\cap \bigcap _{i=1}^{k-1} \rho_{T_i}^{-1}(\tau_{T_i})$ is a disc since 
	$h^{-1}(E_S\cap \bigcap _{i=1}^{k-1} \rho_{T_i}^{-1}(\tau_{T_i}))$ is convex with smooth boundary.	
\end{proof}

We'll eventually need the following, so let's prove it now

\begin{lemma}\label{BS_complement}
	In the above construction, $$B_S-\bigcup_{T\prec S} B_T = \rho_S^{-1}([0,\tau_S])\cap
	\pi_S^{-1}(E_S-\bigcup_{T\prec S}\rho_T^{-1}([0,\tau_T])).$$
\end{lemma}

\begin{proof}
	Suppose first that $x\in \rho_S^{-1}([0,\tau_S])\cap
	\pi_S^{-1}(E_S-\bigcup_{T\prec S}\rho_T^{-1}([0,\tau_T]))$.
	Note that $\beta_S\pi_S(x)=\tau_S$ so $x\in B_S$.
	Suppose that $x\in B_T$ for some $T\prec S$.
	We know $\rho_T(x) \le \tau_T$ and $\pi_S(x)\in U_T$ so $\rho_T\pi_S(x) = \rho_T(x) \le \tau_T$,
	a contradiction.
	So we have shown $x\in B_S-\bigcup_{T\prec S} B_T$.
	
	Conversely, suppose $y\in B_S-\bigcup_{T\prec S} B_T$.
	Since $y\in B_S$ we know $\rho_S(y)\le \tau_S$ and $\pi_S(y)\in E_S$.
	Suppose that $\pi_S(y)\in \rho_T^{-1}([0,\tau_T])$ for some $T\prec S$.
	We may as well suppose $T$ has as small dimension as possible, so
	$\pi_S(y)\not\in \rho_Q^{-1}([0,\tau_Q])$ for all $Q\prec T$.
	We know that $y\in U_T$, so $\rho_T(y) = \rho_T\pi_S(y) \le \tau_T$.
	Take any $Q\prec T$.
	If $\pi_T(y)\not\in U_Q$ then of course $\pi_T(y)\not \in \rho_Q^{-1}([0,\tau_Q])$.
	If  $\pi_T(y)\in U_Q$ and $\rho_Q\pi_T(y)\le \tau_Q$, then $\pi_S(y)\in U_Q$ and thus 
	$\rho_Q\pi_S(y) = \rho_Q\pi_T\pi_S(y) = \rho_Q\pi_T(y)\le \tau_Q$, a contradiction.
	So for all $Q\prec T$ we have $\pi_T(y)\not \in \rho_Q^{-1}([0,\tau_Q])$.
	In particular, $\pi_T(y)\in E_T$ and  $\beta_T(\pi_T(y))= \tau_T$.
	Thus $y\in B_T$, a contradiction.
	So $\pi_S(y)\not\in \bigcup_{T\prec S}\rho_T^{-1}([0,\tau_T])$ and we are done.
\end{proof}

The above construction relies on a well behaved parameterization of the stratum $S$.
	Suppose $h\colon \Delta^k\to X$ is a parameterization of a stratum $S$ of a simplicial stratified set $Z$ in $X$.
	If $T\prec S$, we let $A_{Th}\subset \{0,\ldots,k \}$ denote the vertices of $h^{-1}(T)$, 
	so $h^{-1}(\Cl T) = \mu_{A_{Th}}^{-1}(0)$.

\begin{lemma}\label{simplex_param}
	Suppose $S$ is a stratum of a simplicial stratified set $Z$ in $X$.
	Recall
	$$K_S = S - \bigcup_{T\prec S}\rho_T^{-1}([0,\tau_{T,\dim T})).$$
	Then there is a parameterization $h\colon \Delta^k\to X$ of $S$ so that 
	if $T\prec S$
	and $t\in [\tau_{T,\dim T},\tau_T]$ there is an $s$ so that $h^{-1}(K_S\cap\rho_T^{-1}(t)) \subset \mu_{A_{Th}}^{-1}(s)$.
	In fact, there are $\delta_T>0$ so that if $\Delta' = \Delta^k - \bigcup_{T\prec S}\mu_{A_{Th}}^{-1}([0,\delta_T\tau_{T,\dim T}))$
	then $h^{-1}(K_S) = \Delta'$ and for
	any $t\in [\tau_{T,\dim T},\tau_T]$ we have $h^{-1}(K_S\cap\rho_T^{-1}(t)) = \Delta'\cap \mu_{A_{Th}}^{-1}(t\delta_T)$.
\end{lemma}

\begin{proof}
	Let $S'_\ell = S - \bigcup_{T\prec S, \dim T <\ell}\rho_T^{-1}([0,\tau_{T,\dim T}))$.
	By induction on $\ell$ we'll suppose we have a parameterization $h_\ell\colon \Delta^k\to X$ of $S$ 
	and $\delta_T>0$ for all $T\prec S$ with $\dim T<\ell$
	so that
	\begin{itemize}
		\item $h_\ell^{-1}(S'_\ell) = \Delta^k - \bigcup_{T\prec S, \dim T <\ell}\mu_{A_{Th_\ell}}^{-1}([0,\delta_T\tau_{T,\dim T}))$, and
		\item for any $T\prec S$
		with $\dim T <\ell$
		and any $t\in [\tau_{T,\dim T},\tau_T]$ we have $h_\ell^{-1}(S_\ell'\cap\rho_T^{-1}(t)) = h_\ell^{-1}(S'_\ell) \cap\mu_{A_{Th_\ell}}^{-1}(t\delta_T)$.
		In particular, if $h_\ell(x)\in S'\cap U_T$ and $\rho_Th_\ell(x)\le \tau_T$, then $\rho_Th_\ell(x) = \mu_{A_{Th_\ell}}(x)/\delta_T$.
	\end{itemize} 
	We may start out the induction with $h_0$ any parameterization of  $S$.
	
	I claim there is a smooth vector field $v$ on $\Delta^k$ and an $\epsilon>0$ so that:
	\begin{enumerate}
		\item $v$ is tangent to all faces of $\Delta^k$,
		\item if $h_\ell(x)\in S'_\ell$ and $\mu_{A_{Th_\ell}}(x) \le \delta_T\tau_T$ for some $T\prec S$ with $\dim T<\ell$
		then $d\mu_{A_{Th_\ell}}(v(x))=0$, i.e., $v$ is tangent to fibers $h_\ell^{-1}(S'_\ell \cap \rho_T^{-1}(t))$
		for $\tau_{T,\dim T}\le t \le \tau_T$,
		\item if $h_\ell(x)\in S'_\ell\cap U_T$ for some $T\prec S$ with $\dim T=\ell$
		then $d\rho_T h_\ell(v(x))\ge 0$ and if $d\rho_T h_\ell(v(x))= 0$ then  $v(x)=0$ and 
		$\rho_T h_\ell(x)\not \in (0,\tau_T]$, 
		\item if $h_\ell(x)\in S'_\ell$ and $\mu_{A_{Th_\ell}}(x) \in (0,\epsilon]$ for some $T\prec S$ with $\dim T=\ell$
		then $d\mu_{A_{Th_\ell}}(v(x))>0$, 
		\item if $h_\ell(x)\in S'_\ell$ and $\mu_{A_{Th_\ell}}(x) \le \epsilon$ for some $T\prec S$ with $\dim T=\ell$
		then $h_\ell(x)\in U_T$, and
		\item the support of $v$ is contained in the union of $h_\ell^{-1}(U_T)$ for $T\prec S$ with $\dim T = \ell$.
	\end{enumerate}
	Given the claim we may proceed as follows.  Integrate $v$ to obtain a flow $\phi_s$ on $\Delta^k$.
	By 1 we know $\phi_s$ preserves faces. By 2 we know that if $h_\ell(x)\in S'_\ell$ and $\mu_{A_{Th_\ell}}(x) \le \delta_T\tau_T$ for some $T\prec S$ with $\dim T<\ell$ then $\mu_{A_{Th_\ell}}(\phi_s(x)) = \mu_{A_{Th_\ell}}(x)$ for any $s$
	so in particular $(h_\ell\phi_s)^{-1}(S') = h_\ell^{-1}(S')$ and
	for any $t\in [\tau_{T,\dim T},\tau_T]$ we have $(h_\ell\phi_s)^{-1}(S_\ell'\cap\rho_T^{-1}(t)) = h_\ell^{-1}(S'_\ell) \cap\mu_{A_{Th_\ell}}^{-1}(t\delta_T)$.
	By 3 we know that if $s$ is a large enough positive number, then 
	$$ \phi_s(h_\ell^{-1}(S'_\ell)\cap \mu_{A_{Th_\ell}}^{-1}([0,\epsilon] ))\supset h_\ell^{-1}(S'_\ell)\cap h_\ell^{-1}\rho_T^{-1}([0,\tau_T]).$$
	Pick such an $s$ once and for all.
	
	Choose $\delta_T>0$ so $\delta_T\tau_T\le \epsilon$.
	Now consider a vector field $w$ on $[0,1]\times \Delta^k$ of the form $w(u,x) = (1, \alpha(u,x)v(x))$ for an appropriate 
	smooth function $\alpha$. In particular we want to choose $\alpha$ so that 
	if $$Z= h_\ell^{-1}(S'_\ell)\cap \mu_{A_{Th_\ell}}^{-1}([0,\epsilon])- (\phi_{-s}h_\ell^{-1}\rho_T^{-1}([0,\tau_{T,\dim T}))\cap \mu_{A_{Th_\ell}}^{-1}([0,\delta_T\tau_{T,\dim T})))$$
	and $g_T(u,x) = u\mu_{A_{Th_\ell}}(x)+ (1-u)\delta_T\rho_Th_\ell\phi_s(x)$
	then
	$dg_T(w)=0$ on 
	$[0,1]\times Z$.
	Note 
	\begin{eqnarray*}
		dg_T(w) &= & \mu_{A_{Th_\ell}}(x) - \delta_T\rho_Th_\ell\phi_s(x) + \\
		&&\quad \alpha(u,x)(
		u d\mu_{A_{Th_\ell}}(v(x))+ (1-u)\delta_T d\rho_Th_\ell d\phi_s(v(x))).
	\end{eqnarray*}
	But $d\phi_s(v) = v$ so we may solve 
	$$\alpha(u,x) = (\delta_T\rho_Th_\ell\phi_s(x) - \mu_{A_{Th_\ell}}(x))/(u d\mu_{A_{Th_\ell}}(v(x))+ (1-u)\delta_Td\rho_Th_\ell(v(\phi_s(x)))).$$
	Note $d\rho_Th_\ell(v(\phi_s(x))) > 0$ for $x\in Z$ since $v(\phi_s(x)) = d\phi_s(v) \ne 0$ and consequently
	the denominator is never zero if  $x\in Z$.
	Integrate $w$ to obtain a (partial) flow $\psi_t$ on $[0,1]\times \Delta^k$.
	Note that $(1,f(x)) = \psi_1(0,x)$ defines a diffeomorphism $f\colon \Delta^k\to \Delta^k$
	(it's inverse is defined by  $(0,f^{-1}(x)) = \psi_{-1}(1,x)$).
	Suppose $h_\ell(x)\in S'_\ell\cap U_T$ and $g_T(u,x)\in [\delta_T\tau_{T,\dim T},\delta_T\tau_T]$.
	Then either $\mu_{A_{Th_\ell}}(x)\le \delta_T\tau_T $ or $\delta_T\rho_Th_\ell\phi_s(x)\le  \delta_T\tau_T$
	which in either case implies $\mu_{A_{Th_\ell}}(x)\le \epsilon$.
	Also either $\mu_{A_{Th_\ell}}(x)\ge \delta_T\tau_{T,\dim T}$ or $\delta_T\rho_Th_\ell\phi_s(x)\ge  \delta_T\tau_{T,\dim T}$
	so all in all we must have $x\in Z$ so in particular $w(u,x)$ is tangent to the fiber of $g$ at $(u,x)$.
	Consequently $$f(h_\ell^{-1}(S'_\ell)\cap \phi_s^{-1}h_\ell^{-1}\rho_T^{-1}(t))
	= h_\ell^{-1}(S'_\ell)\cap \mu_{A_{Th_\ell}}^{-1}(\delta_T t)$$
	for all $t\in [\tau_{T,\dim T},\tau_T]$ (since $w$ is everywhere tangent to the fiber  $h_\ell^{-1}(S'_\ell)\cap g^{-1}(\delta_T t)$).
	So we may set $h_{\ell+1} = h_\ell \phi_s f^{-1}$ and the induction is proven.

	To finish we must construct $v$.
	First pick $\epsilon$ small enough that condition 5 holds.
	It suffices to construct $v$ locally and piece together with a partition of unity.
	So pick any $y\in \Delta^k$ and we will construct $v$ in a \nbhd\ of $y$.
	If $h_\ell(y)\not\in S'_\ell$ or $h_\ell(y)\not\in U_T$ for any $T\prec S$ with $\dim T=\ell$
	we may take $v=0$ so we may suppose $h_\ell(y)\in S'_\ell\cap U_T$ for some $T\prec S$ with $\dim T = \ell$.
	Suppose first $h_\ell(y)\in T\cap S'_\ell$.
	After reordering the coordinates, we may as well suppose that $A_{Th_\ell} = \{0,1,\ldots,\ell \}$.
	If $\mu_{A_{Uh_\ell}}(y) \le \delta_U\tau_U$ for any $U\prec S$ with $\dim U < \ell$ we must have
	$h_\ell(y)\in U_U$ and thus
	$U\prec T$ and we may also suppose $0\not\in A_{Uh_\ell}$.
	I claim we may locally pick $v(x) = (-\Sigma_{i=\ell+1}^k x_i,0,\ldots,0,x_{\ell+1},x_{\ell+2},\ldots,x_k)$.
	If $\mu_{A_{Uh_\ell}}(y) \le \delta_U\tau_U$ for any $U\prec S$ with $\dim U < \ell$
	then $\nabla \mu _{A_{Uh_\ell}}$ has 1 in the zeroeth and $\ell+1$ through $k$-th place, so
	$d\mu_{A_{Uh_\ell}}(v) = 0$ and 2 holds.
	Likewise $d\mu_{A_{Th_\ell}}(v) = \Sigma_{i=\ell+1}^k x_i >0$ off of $h_\ell^{-1}(T)$ so 4 holds.
	So the only thing to check is 3.
	For convenience, denote the first $\ell+1$ coordinates of $x$ as $w$ and the last $k-\ell$ coordinates as $z$.
	Then near $h_\ell^{-1}(T)$ we have $\rho_Th_\ell(x) = z^TH(w)z + o(|z|^3)$ where for each $w$, $H(w)$ is a 
	positive definite symmetric $(k-\ell)\times(k-\ell)$ matrix.
	So $\nabla \rho_Th_\ell(x) = (0, 2 H(w)z) + o(|z|^2)$.
	So $d\rho_Th_\ell(v(x)) = \nabla \rho_Th_\ell(x)\cdot v(x) = 2z^TH(w)z + o(|z|^3)$
	which is $>0$ for small enough nonzero $z$.
	So 3 holds.
	
	Now suppose $h_\ell(y)\not\in T\cap K_S$.
	Let $T'$ be the smallest dimensional face of $S$ containing $h_\ell(y)$
	(possibly $T'=S$).
	Suppose $T_1,\ldots,T_m$ are all the faces of $S$ so that  $h_\ell(y)\in U_{T_i}$,
	$\rho_{T_i}h_\ell(y) \le \tau_{T_i}$, and $\dim T_i\le\ell$.
	We may order them so $T_1\prec T_2\prec \cdots\prec T_m= T\prec T'$.
	By Theorem \ref{piconstruct0} we know that the map $x\mapsto (\rho_1(x), \rho_2(x), \ldots, \rho_m(x))$ submerses $T'\cap \bigcap_{i=1}^m U_{T_i}$,
	hence we may find a vector $v_0$ at $y$ tangent to the face $h_\ell^{-1}(T')$ 
	so that $d\rho_ih_\ell(v_0)=0$ for $i<m$ and $d\rho_mh_\ell(v_0)=1$.
	We may locally take our vector field $v$ to be the constant vector filed $v_0$.
	Since $\nabla \rho_ih_\ell$ is constant near $y$  for $i<m$ we will still have $d\rho_ih_\ell(v)=0$ near $y$.
	By continuity, $d\rho_mh_\ell(v_0)>0$ near $y$.
	Note that $v$ will have $j$-th coordinate equal to 0 for all $j\not\in A_{T'h_\ell}$.
	Consequently $v$ will be tangent to all faces of $\Delta^k$ near $h_\ell^{-1}(T')$.
	After possibly replacing $\epsilon$ with a smaller value, 4 will hold.
\end{proof}

\section{Constructing a Tico Prespine}

In the last section we presented a construction which produces a cellular tico from any 
simplicial stratification.  Here we will give a simplicial stratification for which this
construction gives a tico prespine.

We suppose as always that $X$ is a compact smooth manifold with boundary $M$.
Let $c\colon M\times [0,\infty)\to X$ be a smooth collar.
Take any simplical stratification of $X' = X-c(M\times [0,1))$  so that
$M' = c(M\times 1)$ is a union of strata.
For example a smooth triangulation can give such a stratification.

To simplify the proof, for strata $S\subset M'$ we will specify the function $\rho_S$,
as allowed by Theorem \ref{piconstruct0}.
Let $S'\subset M$ be defined by $S = c(S'\times 1)$.
Pick any \MB\ function $\rho'_{S'}$ around the submanifold $S'$ of $M$, and specify $\rho_S$
by the formula $\rho_Sc(x,t) = \rho_{S'}(x) + (t-1)^2$.

Let $\A$ be the cellular tico constructed in the previous section.
Pick a tico map $f\colon (X,|\A|)\to (\R,0)$.
After squaring, we may as well suppose that $f$ is nonnegative.
Let $Z=\bigcup B_S$.
Note that $X-Z\subset c(M\times [0,1))$ because by construction for any stratum $S$,
$S\subset B_S \cup \bigcup_{T\prec S} B_T$.
We will show that if $\epsilon>0$ is small enough, then $dfc(x,t)/dt < 0$ if $c(x,t)\not\in Z$
and $fc(x,t) = \epsilon$.
Recall $U=f^{-1}([0,\epsilon])$ is a regular \nbhd\ of $|\A|$ for small $\epsilon$.
So the vector field $dc(\partial/\partial t)$ always points into $U$ on $\partial U - Z$.
So for a fixed $x\in M$ the curve $c(x,t)$ starts out at $t=0$ in $M$, enters $Z\cup U$ at
some $t=g(x)$ and never leaves $Z\cup U$ afterwards since at a point of leaving we would have to have
 $dfc(x,t)/dt \ge 0$.
So $X-(Z\cup U) = \{c(x,t) \mid x\in M, 0\le t < g(x)\}$
and since the function $g$ is smooth by the inverse function theorem, $\A$ is a tico prespine as desired.

Let $K=\{(x,t)\mid c(x,t)\not\in Z, dfc(x,t)\ge 0 \}$.
We need to show that $\Cl K\cap c^{-1}(|\A|)$ is empty, then we can find an $\epsilon$ as above.
Suppose to the contrary that $(x_0,t_0)\in \Cl K$ and $y_0 = c(x_0,t_0)\in |\A|$.
Let $S_1,\ldots,S_k$ be the sheets of $\A$ which contain $y_0$.
Note all the $S_i$ are contained in $M'$ and we may as well order them so
$S_1\prec S_2\prec \cdots \prec S_k$.
By Lemma \ref{BS_complement} we know that near $(x_0,t_0)$ each $c^{-1}(S_i)$
is given by the equation $\rho_{S'_i}(x) + (t-1)^2 = \tau_{S_i}$.
So near $(x_0,t_0)$,  we have 
$$fc(x,t) = u(x,t)\Pi_{i=1}^k (\rho_{S'_i}(x) + (t-1)^2 - \tau_{S_i})^{b_i}$$
for some exponents $b_i$ and smooth function $u$ with $u(x_0,t_0)>0$.
Also near $(x_0,t_0)$, $c^{-1}(X-Z)$ is given by the inequalities
$\rho_{S'_i}(x) + (t-1)^2 > \tau_{S_i}$ for all $i$.
Then 
$$dfc(x,t)/dt = fc(x,t)(du(x,t)/dt/u+\Sigma_{i=1}^k 2b_i(t-1)/(\rho_{S'_i}(x) + (t-1)^2 - \tau_{S_i})).$$
This must be negative on $c^{-1}(X-Z)$ near $(x_0,t_0)$ since $u'/u$ is bounded, $t-1 < 0$, and the 
denominators $\rho_{S'_i}(x) + (t-1)^2 - \tau_{S_i}$ are positive but near $0$.
Thus we have a contradiction, so we may find $\epsilon$ as desired.

\section{Whitney Appendix}

There is a fancier version of the following where $M$ can have boundary but we don't need it,
so let's keep it relatively simple.
Suppose $X$ is a smooth manifold, $Y\subset X$
is a submanifold, and  $f\colon U\to Z$ is a smooth map where $U\subset X$ is open.
We say $f$ is locally linear with respect to $Y$ if for every $x\in U\cap Y$ we may choose local
coordinates around $x$ and $f(x)$ so that in these coordinates, $f$ is linear, 
 $Y$ is a linear subspace, and $f^{-1}f(x)$ is transverse to $Y$.
 
 \begin{lemma}\label{make_retract}
	Suppose $M$ is a smooth manifold without boundary,
	 $N\subset M$ is a smooth submanifold,
	 $\{U_i\}_{i\in A}$ is a locally finite  collection of open subsets of $M$,
	 $q_i\colon U_i\to Z_i$ are smooth maps
	to manifolds $Z_i$, and  $C_i\subset M$ are closed subsets so that $C_i\subset U_i$. 
	For each nonempty subset
	$D\subset A$ with  $\bigcap_{i\in D} U_i$ nonempty we suppose that the map
	$\Pi_{i\in D}q_i\colon \bigcap_{i\in D} U_i\to \Pi_{i\in D}Z_i$
	is locally linear with respect to $N$.
	Then there is a \nbhd\ $V$ of $N$ in $M$ and a smooth
		retraction $\pi\colon V\to V\cap N$ so that
 $q_i\pi(x)=q_i(x)$ for all $i$ and
			all $x$ in some \nbhd\ of $C_i\cap \pi^{-1}(C_i\cap N)$.
			\end{lemma}
			
\begin{proof}
This follows from the relative version Lemma \ref{make_retract_rel} below with $Y$ empty
and $K=N$.
\end{proof}

\begin{lemma}\label{make_retract_rel}
	Suppose $M$ is a smooth manifold without boundary,
	 $N\subset M$ is a smooth submanifold,
	 $\{U_i\}_{i\in A}$ is a locally finite  collection of open subsets of $M$,
	$q_i\colon U_i\to Z_i$ are smooth maps
	to manifolds $Z_i$, and $C_i\subset M$ are closed subsets so that $C_i\subset U_i$. 
	For each nonempty subset
	$D\subset A$ with  $\bigcap_{i\in D} U_i$ nonempty we suppose that the map
	$\Pi_{i\in D}q_i\colon \bigcap_{i\in D} U_i\to \Pi_{i\in D}Z_i$
	is locally linear with respect to $N$.
	Suppose $Y\subset N$ and $K\subset N$ are closed subsets of $N$,  $U$ is a \nbhd\
	of $Y$ in $M$, and $\sigma\colon U\to U\cap N$ is a smooth 
	retraction  so that 
	$q_i\sigma(x)=q_i(x)$ for all $i$ and
	all $x$ in $U_i\cap \sigma^{-1}(U_i\cap N)$.
	Then there is a \nbhd\ $V$ of $K\cup Y$ in $M$ and a smooth
		retraction $\pi\colon V\to V\cap N$ so that:
		\begin{enumerate}
			\item $\pi(x)=\sigma(x)$ for all $x$ in some \nbhd\ of $Y$.
			\item $q_i\pi(x)=q_i(x)$ for all $i$ and
			all $x$ in some \nbhd\ of $C_i\cap \pi^{-1}(C_i\cap N)$.
			\end{enumerate}
			\end{lemma}
			
			\begin{proof}
				Let $n=\dim N$, $m=\dim M$, and
				for each $i$ let $U_i'$ be a closed \nbhd\ of $C_i$ with  $U_i'\subset U_i$.
				By replacing $M$ with $M-(\Cl N - N)$ we may as well suppose that $N$
				is a closed subset of $M$.

				Suppose $h\colon V_0\to M$ is a parameterization,
				i.e., a smooth embedding from an open subset $V_0\subset \R^m$
				giving local coordinates on $M$.
				Let
				$D=\{ i\mid h(V_0)\cap U'_i\ne \emptyset\}$,
				$U_D=\bigcap _{i\in D} U_i$,
				$Z_D= \Pi_{i\in D} Z_i$,
				and $q_D = \Pi_{i\in D}q_i\colon U_D\to Z_D$.
				For convenience, if $D$ is empty we let $U_D=M$, let $Z_D$ be a point,
				and
				let $q_D\colon M\to Z_D$ be the unique map.
				We
				say $h$ is a good parameterization if $V_0$ is convex and
				there is a 
				parameterization $h_1\colon V_1\to Z_D$ so that
				$h$ and $h_1$ make $q_D$ and $N$ linear,
				in particular so that:
				\begin{itemize}
					\item $V_0$ is convex.
					\item $h(V_0)\subset  U_D$ and 
					$q_D h(V_0)\subset h_1(V_1)$.
					\item $h^{-1}(N) = V_0\cap  E$ for some linear subspace
					$E\subset \R^m$. 
					\item  
					$h_1^{-1}q_D h(x)$ is the restriction of some linear function $f$.
					\end{itemize}
					
					Clearly we may cover $N$ with the images of good parameterizations,
					for any $x$ take any local linear parameterization around $x$ and then delete
					all $U_i'$ which do not intersect $x$, intersect with all other $U_i$, and further 
					restrict to be convex.
					We first suppose that we may cover $K-U$ with a finite number of good parameterizations,
					for example if $K-U$ is compact.
						
					Suppose first only one is needed to cover,
					$K-U\subset h(V_0)$.
					The kernel of $f$ is transverse to $E$, so
					after a linear change of coordinates, we may as well suppose that the kernel of $f$
					contains $E^\perp$.
					Let $g\colon \R^m\to E$ be orthogonal projection.
					Let $\alpha\colon M\to [0,1]$ be a smooth function 
					with support in $U$
					so that  $\alpha$ is 1 on an open \nbhd\ $U'$ of $Y$.
					Define $V= U'\cup h(V_0)$ and
					define $\pi(x)$ to be $\sigma(x)$ for $x\in U'$ and to be 
					$$h((1-\alpha(x))gh^{-1}(x)+\alpha(x)h^{-1}\sigma(x))$$
					for $x\in h(V_0)$.
					Suppose $x$ is near $C_i\cap \pi^{-1}(C_i\cap N)$,
					in particular $x\in U_i'$.
					If $x\in U'$ then $q_i\pi(x) = q_i\sigma(x) = q_i(x)$
					as desired.  So suppose $x = h(y)$ for $y\in V_0$.
					We must have $i\in D$.
					Since $y-g(y)\in E^\perp\subset \ker f$ we must have $fg(y) = f(y)$.
					Since $fh^{-1}\sigma(x) = h_1^{-1}q_D\sigma(x) = h_1^{-1}q_D(x) = f(y)$ 
					we have 
					$$h_1^{-1}q_D\pi(x) = f((1-\alpha(x))g(y)+\alpha(x)h^{-1}\sigma(x))$$
					$$=(1-\alpha(x))fg(y)+\alpha(x)fh^{-1}\sigma(x)$$
					$$=(1-\alpha(x))f(y)+\alpha(x)f(y)=f(y)=h_1^{-1}q_D(x)$$
					and consequently $q_i\pi(x)=q_i(x)$.
					
					Now suppose that $K-U$ is covered by $\ell$ good parameterizations, $K-U\subset \bigcup_{j=1}^\ell h_j(V_j)$.
					Let $W$ be a \nbhd\ of $K-U - \bigcup_{j=1}^{\ell-1} h_j(V_{j})$
					in $N\cap h_{\ell}(V_{\ell})$.
							Note $(K-W)-U$ is covered by $\ell-1$ good parameterizations so by induction on $\ell$ there is a \nbhd\ $V'$
							of $(K-W)\cup Y$ in $M$ and a smooth
							retraction $\pi'\colon V'\to V'\cap N$
							so that
							$\pi'(x)=\sigma(x)$ for all $x$ in some \nbhd\ of $Y$
							and
							$q_i\pi'(x)=q_i(x)$ for all $i$ and
							all $x$ in some \nbhd\ of $C_i\cap \pi'{}^{-1}(C_i\cap N)$.
							But $K-V'\subset W \subset h_{\ell}(V_{\ell})$
							so by the $\ell=1$ case proven above (with $Y$ replaced by $Y\cup(K-U - \bigcup_{j=1}^{\ell-1} h_j(V_{j}))$ ) there is a $\pi\colon V\to V\cap N$
							as desired.
							
							Now we must prove the case where $K-U$ is not compact,
							for this we may as well suppose $K=N$.
							Pick a proper continuous function $\mu\colon M\to [0,\infty)$.
							By this Lemma with $K = N\cap \mu^{-1}([3j-1,3j])$, we have \nbhd s 
							$V_j$ of $(N\cap \mu^{-1}([3j-1,3j]))\cup Y$ and smooth retractions $\sigma_j\colon V_j\to N\cap V_j$  so that 
							\begin{enumerate}
								\item $\sigma_j(x)=\sigma(x)$ for all $x$ in some
								\nbhd\ $V'_j$ of $Y$.
								\item $q_i\sigma_j(x)=q_i(x)$ for all $i$ and
								all $x$ in some \nbhd\ of $C_i\cap \sigma_j^{-1}(C_i\cap N)$.
								\end{enumerate}
								Let $V''_j = V_j\cap \mu^{-1}((3j-2,3j+1))$.
								By this Lemma with 
								\begin{eqnarray*}
									K &=& N\cap \mu^{-1}([3j,3j+2]),\\
									Y &=& Y \cup (N\cap \mu^{-1}([3j-1,3j]\cup [3j+2,3j+3])), \text{\ and}\\
									U &=& (V'_j\cap V'_{j+1})\cup V''_j\cup V''_{j+1}
									\end{eqnarray*}
									we have \nbhd s 
									$W_j$ of $(N\cap \mu^{-1}([3j-1,3j+3]))\cup Y$ and smooth 
									retractions $\tau_j\colon W_j\to N\cap W_j$  so that 
									\begin{enumerate}
										\item $\tau_j(x)=\sigma_j(x)$ for all $x$ in a \nbhd\ of $N\cap \mu^{-1}([3j-1,3j]$.
										\item $\tau_j(x)=\sigma_{j+1}(x)$ for all $x$ in a \nbhd\ of $N\cap \mu^{-1}([3j+2,3j+3]$.
										\item $\tau_j(x)=\sigma(x)$ for all $x$ in a \nbhd\ of $Y$.
										\item $q_i\tau_j(x)=q_i(x)$ for all $i$ and
										all $x$ in some \nbhd\ of $C_i\cap \tau_j^{-1}(C_i\cap N)$.
										\end{enumerate}
										After deleting the closures of 
										$\{x\in W_j\cap V_j\cap \mu^{-1}([3j-1,3j]) \mid \tau_j(x)\ne \sigma_{j}(x)\}$
										and 
										$\{x\in W_j\cap V_{j+1}\cap \mu^{-1}([3j+2,3j+3]) \mid \tau_j(x)\ne \sigma_{j+1}(x)\}$
										and $W_j\cap \mu^{-1}([3j-1,3j]) - V_j$ and
										$W_j\cap \mu^{-1}([3j+2,3j+3]) - V_{j+1}$
										from $W_j$ we may as well also assume that
										\begin{enumerate}
											\item $\tau_j(x)=\sigma_j(x)$ for all $x$ in $W_j\cap \mu^{-1}([3j-1,3j]$.
											\item $\tau_j(x)=\sigma_{j+1}(x)$ for all $x$ in $W_j\cap \mu^{-1}([3j+2,3j+3]$.
											\end{enumerate}
											Now set 
											$$V = \bigcup_{j=0}^\infty W_j\cap \mu^{-1}((3j-1,3j+3))
											$$
											and define $\pi\colon V\to N$ by $\pi(x) = \tau_j(x)$
											for $x\in{W_j\cap \mu^{-1}((3j-1,3j+3))}$.											\end{proof}

\begin{lemma}\label{Whitney_submersion}
	Suppose $S,T$ are disjoint  submanifolds of $M$ satisfying the Whitney conditions, 
	$U$ is a \nbhd\ of $S$ in $M$,
	$\rho\colon U\to [0,\infty)$ is \MB\ around $S$,
	and $\pi\colon U\to S$ is a smooth retraction.
	Then there is a \nbhd\ $U'$ of $S$ in $U$ so that
	$\rho\times \pi|\colon U'\cap T\to (0,\infty)\times S$
	is a submersion.
\end{lemma}

\begin{proof}
	If not, there is a sequence of points $z_i\in T$, $i=1,\ldots$
	so that $z_i\to z\in S$
	and $\rho\times \pi$ restricted to $T$ has rank
	$< 1 + \dim S$ at each $z_i$.
	Choose a smooth parameterization
	$p\colon V\to M$ of a \nbhd\ of $z$
	so that $V\subset \R^m$ and 
	$p^{-1}S=V\cap  \bigcap_{i=1}^k\R^m_i$.
	Let $y_i=p^{-1}(z_i)$ and $y=p^{-1}(z)$.
	Define $\pi_0\colon \R^m\to \R^m$, $\pi_1\colon \R^m\to \R^m$,
	$\pi_0'\colon \R^m\to \R^k$,
	and $\pi_1'\colon \R^m\to \R^{m-k}$
	by $\pi_0(x) = (x_1,\ldots,x_k,0,\ldots,0)$,
	$\pi_1(x)=(0,\ldots,0,x_{k+1},\ldots,x_m)$,
	$\pi'_0(x) = (x_1,\ldots,x_k)$,
	and
	$\pi_1'(x)=(x_{k+1},\ldots,x_m)$.
	We may assume after taking a subsequence that the tangent
	spaces to $p^{-1}T$ at $y_i$ converge to some subspace $E$ of
	$\R^{m}$ and 
	$\pi_0(y_i)/|\pi_0(y_i)|$ converge to some
	unit vector $v$.
	
	By the Whitney conditions we know $v\in E$ and
	$0\times \R^{m-k}\subset E$.
	Let $e_i\in \R^m$ be the vector with $i$-th coordinate 1 and $j$-th coordinate 0 for $j\ne i$.
	Then there are vectors $w_{ij}\in \R^m$ and $w_j\in \R^m$
	so that $\pi_0(y_j)/|\pi_0(y_j)|+w_j$ and $e_i+w_{ij}$ are in the tangent space
	to $p^{-1}T$ at $y_j$ for all $j$ and $i=k+1,\ldots,m$
	and furthermore $\lim_{j\to \infty}w_j=0$ and 
	$\lim_{j\to \infty}w_{ij}=0$ for all $i=k+1,\ldots,m$.
	
	Let $f\colon V\to \R\times \R^{m-k}$ be the function
	$f(x) = (\rho p,\pi_1'p^{-1}\pi p)$.
	By Lemma \ref{division} we know there are a smooth 
	$k\times k$ matrix valued function $L\colon V\to \R^{k\times k}$
	and a smooth $(m-k)\times k$ matrix valued function 
	$K\colon V\to \R^{(m-k)\times k}$
	so that $\rho p(x) = \pi_0'(x)^TL(x)\pi_0'(x)$
	and $\pi_1'p^{-1}\pi p(x) = \pi_1'(x) + K(x)\pi_0'(x)$.
	Replacing $L(x)$ by $(L(x)+L(x)^T)/2$ we may as well suppose that
	$L(x)$ is always symmetric.
	
	Let $J(x)$ be the Jacobian of $f$ at $x$.
	By assumption, the $1+\dim S$ vectors $J(y_j)(\pi_0(y_j)/|\pi_0(y_j)|^2+w_j/|\pi_0(y_j)|)$
	and $J(y_j)(e_i+w_{ij})$ for $i=k+1,\ldots,m$ are linearly dependent for each $j$.
	Note that $\lim_{j\to\infty}J(y_i)(e_i+w_{ij})=J(y)(e_i)=
	\begin{bmatrix}
	0\\
	\pi_1'(e_i)
	\end{bmatrix}$ for $i>k$.
	On the other hand, the first row $J_0(x)$ of $J(x)$
	is $[2\pi'_0(x)^TL(x)\ 0]+\Sigma_{i=1}^k\Sigma_{\ell=1}^k x_ix_\ell J'_{i\ell}(x)$ for some smooth row vector valued
	functions $J'_{i\ell}$.
	Then 
	\begin{eqnarray*}
		&&\lim_{j\to\infty}J_0(y_j)
		( \pi_0(y_j)/|\pi_0(y_j)|^2+w_j/|\pi_0(y_j)|)\\
		&=& \lim_{j\to\infty}[2\pi'_0(y_j)^TL(y_j)\ 0]
		( \pi_0(y_j)/|\pi_0(y_j)|^2+w_j/|\pi_0(y_j)|)\\
		&=& \lim_{j\to\infty} 2 \pi'_0(y_j)^TL(y_j) \pi'_0(y_j)/| \pi_0(y_j)|^2 + 2 \pi'_0(y_j)^TL(y_j)\pi'_0(w_j)/| \pi_0(y_j)| \\
		&=& 2v^TL(y)v + v^TL(y)0 = 2v^TL(y)v >0
	\end{eqnarray*}
	and consequently 
	$\lim_{j\to\infty} J(y_j)(\pi_0(y_j)/|\pi_0(y_j)|^2+w_j/|\pi_0(y_j)|) = \begin{bmatrix}
	2v^TL(y)v\\
	*
	\end{bmatrix}$.
	So we have a sequence of linearly dependent vectors
	with linearly independent limits, a contradiction.
\end{proof}

\begin{lemma}\label{goodshrink}
	Suppose $U$ is an open \nbhd\ of a submanifold $N$ of a manifold $M$.
	Then there is an open \nbhd\ $V$ of $N$ so that $\Cl V \subset U\cup \Cl N$.
	We call such a $V$ a {\em good shrinking} of $U$.
\end{lemma}

\begin{proof}
	Let $X = \Cl N -N$, $M' = M-X$,
	and $C =M' -U$.
	Then $C$ and $N$ are disjoint
	closed subsets of the normal space $M'$ and hence there
	is an open \nbhd\ $V$ of $N$ in $M'$ whose closure
	in $M'$ does not intersect $C$.
	Then the closure of $V$ in $M$ is contained in $U\cup X$.
\end{proof}

\begin{lemma}\label{proper_rho}
	Suppose $U$ is an open \nbhd\ of a submanifold $N$ of a manifold $M$,
	$\pi\colon U\to N$ is a smooth retraction, and $\rho\colon U\to [0,\infty)$ is \MB\ around $N$.
	Then there is an open \nbhd\ $V$ of $N$ so that:
	\begin{itemize}
		\item $V\subset U$.
		\item $V$ is a good shrinking of $U$, i.e., $\Cl V\subset U\cup \Cl N$.
		\item $\pi\times\rho\colon V\to W$ is a proper map, where $W = (\pi\times\rho)(V)\subset N\times [0,\infty)$.
		\item The restriction of $\pi\times \rho$ to $V-N$ is a proper submersion onto $W-N\times 0$
		(and hence by a theorem of Ehresmann is a smoothly locally trivial fibration\footnote{in fact a sphere bundle}).
		\item There is a smooth $\gamma\colon N\to (0,\infty)$ so that 
		$$W=\{(x,t)\in N\times [0,\infty) \mid t < \gamma(x) \}.$$
	\end{itemize}
\end{lemma}

\begin{proof}
	By Lemma \ref{Whitney_submersion} with $T=U-N$ and $S=N$ we may as well suppose that
	$\pi\times\rho\colon U-N\to N\times (0,\infty)$ is a submersion.
	After further shrinking $U$, we may as well also suppose that for any compact $K\subset N$
	that $\Cl(\pi^{-1}(K))$ is compact, for example intersect with the interior
	of a closed tubular \nbhd\ of $N$ or use a more elementary argument.
	By Lemma \ref{goodshrink} we may choose a good shrinking $V'$ of $U$.
	For each $x\in N$ choose an open \nbhd\ $Z_x$ of $x$ in $N$ so that
	$\Cl Z_x$ is compact.
	By compactness we may choose $\alpha_x>0$  so that $\rho(y)\ge \alpha_x$ for all $y\in (\Cl V' - V')\cap \pi^{-1}(Z_x)$.
	Take a partition of unity $\{\beta_x\}$ for the open cover $\{Z_x\}$ of $N$ and let $\gamma(z) = \Sigma_{x\in N} \beta_x(z) \alpha_x$.
	Note that for any $z\in N$, $\gamma(z) \le \rho(y)$ for all $y\in (\Cl V'-V')\cap \pi^{-1}(z)$
	since $\gamma(z)\le \max_{z\in Z_x}\alpha_x$.
	Let $W= \{(x,t)\in N\times [0,\infty) \mid t < \gamma(x) \}$ and $V = (\pi\times\rho)^{-1}(W)\cap V'$. 
	Note $\Cl V\subset \Cl V' \subset U\cup \Cl N$ so we only need to show that
	$\pi\times\rho\colon V\to W$ is proper.
	
	Take $K\subset W$ compact.
	We must show that $(\pi\times \rho)^{-1}(K)\cap V$ is compact.
	We know $K\subset K'\times [0,m]$ for some compact $K'\subset N$ and some $m$.
	Then $(\pi\times \rho)^{-1}(K)\cap \Cl V$ is a closed subset of $\pi^{-1}(K')\cap \Cl V$
	which is itself a closed subset of the compact $\Cl\pi^{-1}(K')$.
	Hence $(\pi\times \rho)^{-1}(K)\cap \Cl V$ is compact.
	But $(\pi\times \rho)^{-1}(K)\cap V = (\pi\times \rho)^{-1}(K)\cap \Cl V$
	since $\rho(x)\ge \gamma\pi(x)$ (in fact $=$) for $x\in \Cl V-V$.	
\end{proof}

\begin{lemma}\label{piconstruct}
	Let $X$ be a Whitney stratified subset of a smooth manifold $M$,
	$q\colon M\to N$ is a smooth submersion to a smooth manifold $N$ which restricts to a submersion on each stratum of $X$,
	and suppose that for each stratum $S$ of $X$ we choose a
	\MB\
	function $\rho_S\colon U''_S\to [0,\infty)$
	around  $S$. 
	Moreover, suppose that there is a closed set $Y\subset X$ and for each stratum
	$S$ we have a (possibly empty) \nbhd\ $U'_S$ of $Y\cap S$ in $U''_S$ and 
	a smooth retraction $\pi'_S\colon U'_S\to S\cap U'_S$ 
	so that $q\pi'_S=q$ and if $T\prec S$ then $\pi'_T\pi'_S = \pi'_T$ and
		$\rho_T\pi'_S = \rho_T$
		on $\pi'_S{}^{-1}(U'_T)\cap U'_T$.
	Then there are good shrinkings $U_S$ of the \nbhd s $U''_S$ and
	smooth retractions $\pi_S\colon U_S\to S$ for all strata $S$ so that
	\begin{enumerate}
		\item If $T\prec S$ then $U_S\cap U_T= \pi_S^{-1}(S\cap U_T)$.
		\item If $T\prec S$ then $\pi_T\pi_S = \pi_T$ and
		$\rho_T\pi_S = \rho_T$
		on $U_S\cap U_T$.
		\item If $U_S\cap U_T$ is nonempty, then either $S\prec T$,
		$T\prec S$, or $S=T$.
		\item If $T_1\prec T_2\prec\cdots\prec T_\ell\prec S$ then
		$$(\rho_{T_1},\rho_{T_2},\ldots,\rho_{T_\ell},\pi_{T_1})
		\colon S\cap U_{T_1}\cap\cdots\cap U_{T_\ell} \to
		(0,\infty)^\ell\times T_1$$ is a submersion.
		\item $\pi_S\times \rho_S\colon U_S\to W_S$ is a proper map onto $W_S$, where
		$W_S=\{(x,t)\in S\times [0,\infty) \mid  t < \gamma_S(x) \}$
		for some smooth $\gamma_S\colon S\to (0,\infty)$.
		\item Any compact subset of $X$ intersects only a finite number of the $U_S$.
		\item $\pi_S = \pi'_S$ on a \nbhd\ of $Y\cap S$ in $X$.
		\item $q\pi_S=q|_{U_S}$.
		\end{enumerate}
\end{lemma}

\begin{proof}
	Let $\emptyset = K_0\subset K_1\subset K_2\subset \cdots$ be a sequence of compact subsets of $M$
	whose union is $M$.
	For each stratum $T$ let $j$ be the largest integer so that $T\cap K_j$ is empty.
	Replace $U''_T$ by the smaller open \nbhd\ $U''_T-K_j$.
	By local finiteness, any compact subset intersects only a finite number of strata $S$,
	and hence only a finite number of $U''_S$.  So the sixth condition will hold.	
	
	The third condition is easily obtained by shrinking the $U''_S$.
	We suppose by induction on $i$ that if $S$ and $T$ are strata which intersect $K_i$ and neither $T\prec S$
	nor $S\preceq T$ then $U''_S\cap U''_T$ is empty.
	The inductive step shrinks the $U''_S$ leaving $U''_S\cap K_i$ fixed.
	In the end, each $U''_S$ will still be open even though it may have been shrunk infinitely often.

	By induction on $k$ we may suppose that for all strata $S$ of dimension $<k$ we have a good shrinking
	$U_S$ of $U''_S$ and a retraction $\pi_S\colon U''_S\to S$ so that:
	\begin{itemize}
		\item If $T\prec S$ and $\dim S<k$, then $U_S\cap U_T= \pi_S^{-1}(S\cap U_T)\cap U_S$.
		\item If $T\prec S$ and $\dim S<k$, then $\pi_T\pi_S(x) = \pi_T(x)$ and
		$\rho_T\pi_S(x) = \rho_T(x)$ for all $x\in U_T\cap \pi_S^{-1}(S\cap U''_T)$.
		\item If $T_1\prec T_2\prec\cdots\prec T_\ell\prec S$ and $\dim T_\ell <k$, then
		$$(\rho_{T_1},\rho_{T_2},\ldots,\rho_{T_\ell},\pi_{T_1})
		\colon S\cap U_{T_1}\cap\cdots\cap U_{T_\ell} \to
		(0,\infty)^\ell\times T_1$$ is a submersion.
		\item If $\dim S<k$ then $\pi_S\times \rho_S\colon U_S\to W_S$ is a proper map onto $W_S$, where
		$W_S=\{(x,t)\in S\times [0,\infty) \mid  t < \gamma_S(x) \}$
		for some smooth $\gamma_S\colon S\to (0,\infty)$.
		\item If $\dim S<k$ then $\pi_S = \pi'_S$ on a \nbhd\ of $Y\cap S$ in $X$.
		\item If $\dim S<k$ then  $q\pi_S=q|_{U_S}$.
	\end{itemize}
	Now let $S$ be any stratum of dimension $k$.  For the inductive step we need an appropriate retraction $\pi_S\colon U''_S\to S$,
	after perhaps shrinking $U''_S$, as well as a good shrinking $U_S$ of $U''_S$.
		This follows from Lemma \ref{make_retract_rel}.
		In particular, we let $Z=\Cl S - S = \bigcup_{T\prec S}T$
		and for each $T\prec S$ we let $V_T$ be a good shrinking of $U_T$
		satisfying the conclusions of Lemma \ref{proper_rho}.
		We apply Lemma \ref{make_retract_rel} with $N= S$, $M= M-Z$, $K=S$, $Y = Y\cap S$, and $\sigma = \pi'_S$.
		The $U_i$ are $U_T-Z$ for $T\prec S$, the $C_i$ are $\Cl V_T - Z$, and the $q_i$ are 
		$\pi_T\times \rho_T\colon U_T-Z\to T\times \R$.
		One additional $U_i$ is $U_0=C_0=M$ and $q_0 = q$.
		To satisfy the hypotheses of Lemma \ref{make_retract_rel} we must show that if
		$T_1,\ldots,T_\ell$ is a collection of strata with $T_i\prec S$ and $\bigcap_{i=1}^\ell U_{T_i}$ nonempty 
		then $q\times \Pi_{i=1}^\ell \pi_{T_i}\times \Pi_{i=1}^\ell \rho_{T_i}\colon
		\bigcap_{i=1}^\ell U_{T_i}\to \Pi_{i=1}^\ell T_i \times \R^\ell$
		is locally linear
		with respect to $S$.
		We may as well reorder so that $T_1\prec T_2\prec \cdots\prec T_\ell$.
		Take any $z\in S\cap\bigcap_{i=1}^\ell U_{T_i}$.
		Take any parameterization $h_0\colon V_0 \to N$ of a \nbhd\ of $q(z)=q\pi_{T_1}(z)$ in $N$ with $V_0\subset \R^{m_0}$ open.
		Since $q|_{T_1}$ is a submersion there is a parameterization
		 $h_1\colon V_1\to T_1$
		of a \nbhd\ of $\pi_{T_1}(z)$
		with $V_1\subset \R^{m_1}$ open so that $h_0^{-1}qh_1(x) = (x_1,\ldots,x_{m_0})$.
		Since $(\pi_{T_1},\rho_{T_1})|\colon T_2\cap U_{T_1} 
		\to T_1\times \R$ is a  submersion, we may choose 
		a parameterization $h_2\colon V_2\to T_2$ of a \nbhd\ of $\pi_{T_2}(z)$
		with 
		$V_2\subset \R^{m_2}$ open
		so that $h_1^{-1}\pi_{T_1}h_2(x) = (x_1,\ldots,x_{m_1})$
		and $\rho_{T_1}h_2(x) = x_{m_1+1}$.
		Continuing in this way we get for each $i=1,\ldots,\ell$
		a parameterization $h_i\colon V_i\to T_i$ of a \nbhd\ of $\pi_{T_i}(z)$
		with 
		$V_i\subset \R^{m_i}$ open,
		$h_{i}^{-1}\pi_{T_{i}}h_{i+1}(x) = (x_1,\ldots,x_{m_{i}})$,
		and $\rho_{T_i}h_{i+1}(x) = x_{m_i+1}$.
		Since $\pi_{T_\ell}\times \rho_{T_\ell}$ submerses the pair $(M,S)$ at $z$ we may 
		likewise choose a parameterization $h_{\ell+1}\colon V_{\ell+1}\to M$ of a \nbhd\ of $z$ in $\bigcap_{i=1}^\ell U_{T_i}$
		with 
		$V_{\ell+1}\subset \R^{m_{\ell+1}}$ open,
		$h_\ell^{-1}\pi_{T_\ell}h_{\ell+1}(x) = (x_1,\ldots,x_{m_\ell})$,
		$\rho_{T_\ell}h_{\ell+1}(x) = x_{m_\ell+1}$, $h_{\ell+1}^{-1}(z)=0$, and 
		$h_{\ell+1}^{-1}(S) = V_{\ell+1}\cap  E$
		for the linear subspace $E$ defined by the vanishing of $x_i$ for all $i>\dim S$.
		Note that $\pi_{T_i}=\pi_{T_i}\pi_{T_{i+1}}\cdots \pi_{T_\ell}$ and
		 $\rho_{T_i}=\rho_{T_i}\pi_{T_{i+1}}\cdots \pi_{T_\ell}$
		and thus $h_{i}^{-1}\pi_{T_{i}}h_{{\ell+1}}(x) = (x_1,\ldots,x_{m_{i}})$
		and $\rho_{T_i}h_{{\ell+1}}(x) = x_{m_i+1}$ for all $i=1,\ldots,\ell$.
		Since $q\pi_{T_1}=q$ we also know $h_0^{-1}qh_{\ell+1}(x) = (x_1,\ldots,x_{m_{0}})$.
		Consequently, each coordinate of
		$(\Pi_{i=1}^\ell h_i\times {\rm identity}\times h_0)^{-1}
		(\Pi_{i=1}^\ell \pi_{T_i}\times \Pi_{i=1}^\ell \rho_{T_i}\times q)h_{\ell+1}(x)$
		is some coordinate of $x$, in particular linear.
		We know $(\Pi_{i=1}^\ell \pi_{T_i}\times \Pi_{i=1}^\ell \rho_{T_i}\times q)^{-1}(\Pi_{i=1}^\ell \pi_{T_i}\times \Pi_{i=1}^\ell \rho_{T_i}\times q)(z)$ is transverse to $S$ since $h_{\ell+1}^{-1}$ of it contains $E^\perp$.
		So $\Pi_{i=1}^\ell \pi_{T_i}\times \Pi_{i=1}^\ell \rho_{T_i}\times q$
		is locally linear with respect to $S$.
		So by Lemma \ref{make_retract_rel}
		there is a \nbhd\ $U''_S$ of $S$ in $M-Z$ and a smooth retraction
		$\pi_S\colon U''_S\to S$ so that  $\pi_S(x)=\pi'_S(x)$ for all $x$ in a \nbhd\ of $Y\cap S$,
		$q\pi_S(x)=q(x)$,
		$\pi_{T}\pi_S(x)=\pi_{T}(x)$, and 
		$\rho_{T}\pi_S(x)=\rho_{T}(x)$ if $T\prec S$ and $x$ is in some \nbhd\ $V'_T$ of
		 $\Cl V_T\cap \pi_S^{-1}(S\cap \Cl V_T)$.		
		At this point, replace each $U_T$ by $V_T$
		and replace $U''_T$ by some \nbhd\ $V''_T$ of $\Cl V_T\cap U''_T$ in $U''_T$ so that
		$V_T\cap \pi_S^{-1}(S\cap V''_T)\subset V'_T$, for example choose any \nbhd\ $V''_T$ of $\Cl V_T\cap U''_T$ in $U''_T$
		so that $S\cap \Cl V''_T\subset V'_T$ and delete from $U''_S$ the closure of $\pi_S^{-1}(V''_T\cap S)-V'_T$.
		Note that there are only finitely many strata $P$ with $T\prec P$ so each $U_T$ is only reshrunk finitely often.
		
		We now want to shrink $U''_S$ a bit more if needed so that for each $T\prec S$,
		$\pi_S(V_T\cap U''_S)\subset V''_T$. It suffices to find for each $x\in S$ a \nbhd\ $W$ of $x$ in $U''_S$
		so that $\pi_S(V_T\cap W)\subset V''_T$ for all $T\prec S$.
		We start with a \nbhd\ $W'$ of $x$ in $U''_S$ which intersects only finitely many $V''_T$, any \nbhd\ with compact closure will do.
		We obtain $W''$ by intersecting $W'$ with all $ \pi_S^{-1}(S\cap V''_T)$ such that $x\in V''_T$.
		We obtain $W$ by deleting $\Cl V_T$ from $W''$ for all $T$ so that $W''$ intersects $V''_T$ but $x\not\in \Cl V_T$.
		Then $\pi_S(V_T\cap W)\subset V''_T$ for all $T\prec S$ since the left hand side is empty unless $x\in V''_T$
		in which case $\pi_S(W)\subset V''_T$.

	We let $U_S$ be any good shrinking of $U''_S$ satisfying the conclusions of Lemma 	\ref{proper_rho}.
	Now let us see why the inductive conditions are satisfied.
	The second and the last three conditions are immediate.
	Let us show the first condition. 
		Suppose $x\in \pi_S^{-1}(S\cap V_T)\cap U_S-V_T$.
		By the path lifting property for fibrations or by integrating a suitable vector field
		we may choose a continuous path
		$\beta\colon (0,\rho_S(x)]\to U_S$ so that $\pi_S\beta(t) = \pi_S(x)$ and
		$\rho_S\beta(t) = t$ for all $t$ and $\beta(\rho_S(x)) = x$. 
		We may continuously extend $\beta$ to $0$ by setting $\beta(0)=\pi_S(x)$.
		Let $t_0 = \inf\{t\mid \beta(t)\not\in V_T \}$.
		For $t<t_0$ we know $\beta(t)\in V_T\cap U_S$ so 
		$\pi_T\beta(t) = \pi_T\pi_S\beta(t) = \pi_T\pi_S(x)$ and
		$\rho_T\beta(t) = \rho_T\pi_S\beta(t) = \rho_T\pi_S(x)$
		are constant.
		But by properness of $\rho_T\times \pi_T|_{V_T}$ we then know $\beta(t_0)\in V_T$
		so $\beta(t)\in V_T$ for $t$ slightly larger than $t_0$, a contradiction.
		So $\pi_S^{-1}(S\cap V_T)\subset U_S\cap V_T$.
		Now suppose $x\in U_S\cap V_T$ but $\pi_S(x)\not\in V_T$.
		Again we may choose a continuous path 
			$\beta\colon [0,\rho_S(x)]\to U_S$ so that $\pi_S\beta(t) = \pi_S(x)$ and
			$\rho_S\beta(t) = t$ for all $t$ and $\beta(\rho_S(x)) = x$. 
		Let $t_0 = \sup \{ t\mid \pi_S\beta(t)\not \in V_T  \}$.
		Recall that we shrunk $U''_S$ enough to guarantee that $\pi_S\beta(t)\in V''_T$ for $t>t_0$
		so $\pi_T\beta(t)=\pi_T\pi_S\beta(t)=\pi_T\pi_S(x) = \pi_T(x)$ and 
		$\rho_T\beta(t) = \rho_T\pi_S\beta(t) = \rho_T\pi_S(x)= \rho_T(x)$.
		So again $\beta(t_0)\in V_T$ and we get a contradiction.
		So $\pi_S^{-1}(S\cap V_T) = V_T\cap U_S$.
		A similar argument shows that $\pi_P^{-1}(P\cap V_T)\cap V_P = V_T\cap V_P$ if $T\prec P$ and $\dim P<k$.

	For the third condition, suppose we have $T_1\prec\cdots\prec T_\ell\prec S\prec P$ then we must show that
	$\rho_S\times(\rho_{T_1},\rho_{T_2},\ldots,\rho_{T_\ell})\times\pi_{T_1}
	\colon P\cap V_{T_1}\cap\cdots\cap V_{T_\ell}\cap U_S \to
	(0,\infty)\times(0,\infty)^{\ell}\times T_1$ is a submersion.
	Note this map is the composition of the submersions
	$\rho_S\times \pi_S\colon P\cap V_{T_1}\cap\cdots\cap V_{T_\ell}\cap U_S \to 	(0,\infty)\times S$
	and $id\times(\rho_{T_1},\rho_{T_2},\ldots,\rho_{T_\ell},\pi_{T_1})$.
	
	After we have finished with the induction, we just restrict each $\pi_S$ and $\rho_S$ to $U_S$ and we are done.
\end{proof}

\end{document}